\documentclass[12pt,a4paper,draft]{amsart}
\usepackage[top=3cm,bottom=4cm,left=2.9cm,right=2.9cm,footskip=17mm]{geometry}
\usepackage{array,float,mathtools,verbatim}
\usepackage{amssymb,hyperref ,amsmath,amsthm,amsbsy,amscd, mathrsfs}

\numberwithin{equation}{section}
\numberwithin{figure}{section}
\usepackage[shortlabels]{enumitem}
\usepackage{color}
\usepackage[normalem]{ulem}

\newtheorem{theorem}{Theorem}[section]
\newtheorem{corollary}[theorem]{Corollary}
\newtheorem{lemma}[theorem]{Lemma}
\newtheorem{proposition}[theorem]{Proposition}

\newtheorem{definition}[theorem]{Definition}

\makeatletter
\newcommand{\imod}[1]{\allowbreak\mkern4mu({\operator@font mod}\,\,#1)}
\makeatother

\newcommand{\suchthat}{\;\ifnum\currentgrouptype=16 \middle\fi|\;}

\newtheorem{thmABC}{Theorem}

\newtheorem{corABC}[thmABC]{Corollary}

\DeclareMathOperator{\Ind}{Ind}
\DeclareMathOperator{\Res}{Res}
\DeclareMathOperator{\rea}{Re}
\DeclareMathOperator{\GL}{GL}
\DeclareMathOperator{\Spec}{Spec}

\def \bfa {{\bf a}}
\def \bfb {{\bf b}}
\def \bfy {{\bf y}}
\def \bfY {{\bf Y}}
\def \bfG {{\bf G}}
\def \bfs {{\bf s}}
\def \bfP {{\bf P}}
\def\wt{\widetilde}
\def\bb{\mathbb}
\def \Fq{\mathbb{F}_q}
\def\bf{\mathbf}
\def\cal{\mathcal}
\def\frak{\mathfrak} 

\def\CC{\bb{C}}
\def\ZZ{\bb{Z}}

\def\QQ{\bb{Q}}
\def\NN{\bb{N}}
\def\ZZ{\bb{Z}}

\def\RR{\bb{R}}
\def\T{\mathcal{T}}
\def\Irr{\widetilde{\mathrm{Irr}}} 
\def\diag{\mathrm{diag}} 
\def\Mat{\mathrm{Mat}} 
\def\rm{\textrm}
\def\R{\widetilde{R}}

\def\a{\alpha}
\def\aO{\a(\G(\O))}

\def\k{\kappa}
\def\i{\iota}
\def\g{\mathfrak{g}} 
\def\p{\mathfrak{p}} 

\def\o{\mathfrak{o}} 
\def\O{\mathcal{O}} 

\def\G{\mathbf{G}} 
\def\H{\mathbf{H}}
\def\Zpadic{
\int_{(x,{\bf y})\in\p\times (\o^d)^*} |x|^\tau_{\p}
\prod_{j=1}^u\frac{\|{F}_j({\bf y})\cup {F}_{j-1}({\bf y})x^2\|^\rho_{\p}}
{\|{F}_{j-1}({\bf y})\|_{\p}^\rho}
\prod_{i=1}^v\frac{\|{G}_i({\bf y})\cup {G}_{i-1}({\bf y})x\|_{\p}^\sigma}
{\|{G}_{i-1}({\bf y})\|_{\p}^\sigma}d\mu(x,{\bf y})
}

\def\ZIpadic{
\int_{(x,{\bf y})\in\p\times\GL_d(\o)} \prod_{\k\in[l]}\left\|
\bigcup_{\i\in I_\k}x^{e_{\k\i}}F_{\k\i}({\bf y})
\right\|_\p^{s_\k}d\mu(x,{\bf y})
}

\theoremstyle{definition}

\newtheorem{example}[theorem]{Example}
\newtheorem{remark}[theorem]{Remark}

\author{Duong H.\ Dung and Christopher Voll}
\address{Fakult\"{a}t f\"{u}r Mathematik\\
Universit\"{a}t Bielefeld\\
Postfach 100131\\
D-33501 Bielefeld, Germany.}
\email{dhoang@math.uni-bielefeld.de, voll@math.uni-bielefeld.de}


\date{\today}


\begin{document}
\title[Representation zeta functions of nilpotent groups]{Uniform
  analytic properties of representation zeta functions of finitely
  generated nilpotent groups}


\begin{abstract}
  Let $G$ be a finitely generated nilpotent group. The representation
  zeta function $\zeta_G(s)$ of $G$ enumerates twist isoclasses of
  finite-dimensional irreducible complex representations of $G$. We
  prove that $\zeta_G(s)$ has rational abscissa of convergence $\a(G)$
  and may be meromorphically continued to the left of $\a(G)$ and
  that, on the line $\{s\in\CC \mid \rea(s) = \a(G)\}$, the continued
  function is holomorphic except for a pole at~$s=\a(G)$.  A Tauberian
  theorem yields a precise asymptotic result on the representation
  growth of $G$ in terms of the position and order of this pole.
  
  We obtain these results as a consequence of a result
  establishing uniform analytic properties of representation zeta
  functions of torsion-free finitely generated nilpotent groups of the
  form $\G(\O)$, where $\G$ is a unipotent group scheme defined in
  terms of a nilpotent Lie lattice over the ring $\O$ of integers of a
  number field. This allows us to show, in particular, that the
  abscissae of convergence of the representation zeta functions of
  such groups and their pole orders are invariants of~$\G$,
  independent of~$\O$.
\end{abstract}

\maketitle
\thispagestyle{empty}
\section{Introduction}
\subsection{Main results}

In this paper we study zeta functions associated to finitely generated
nilpotent groups. Let $G$ be such a group and let $\rho$, $\sigma$ be
complex representations of $G$. We say that $\rho$ and $\sigma$ are
\emph{twist-equivalent} if there exists a $1$-dimensional
representation $\chi$ of $G$ such that
$\rho\otimes\chi\cong\sigma$. Twist-equivalence is an equivalence
relation on the set of finite-dimensional irreducible complex
representations of $G$. Its classes are called \emph{twist
  isoclasses}. The number $\widetilde{r}_n(G)$ of twist isoclasses of
$G$ of dimension $n$ is finite for every $n\in\NN$;
cf.\ {\cite[Theorem 6.6]{Lubotzky-Magid}}.

The (\textit{representation}) \emph{zeta function} of $G$ is defined to be the
Dirichlet generating function
$$\zeta_G(s):=\sum_{n=1}^\infty{\widetilde{r}_n(G)}{n^{-s}},$$ where
$s$ is a complex variable. The sequence
$\left(\widetilde{r}_n(G)\right)_{n\in\NN}$ grows polynomially and
hence $\zeta_G(s)$ converges on a complex half-plane $\{s\in\bb{C}
\mid \rea(s)>\alpha\}$ for some $\alpha\in\bb{R}$;
cf.\ {\cite[Lemma~2.1]{SV/14}. The infimum of such $\alpha$ is the
  \emph{abscissa of convergence} $\alpha(G)$ of~$\zeta_G(s)$. Note
  that if $\wt{r}_n(G)\neq 0$ for infinitely many $n\in\NN$, then
$$\a(G)=\limsup_{N\rightarrow\infty}\frac{\log\sum_{n=1}^N\widetilde{r}_n(G)}{\log
    N}.$$ Hence $\a(G)$ determines the rate of polynomial growth
  of~$\left(\sum_{n=1}^N\widetilde{r}_n(G)\right)_{N\in\NN}$.

  We consider a class of torsion-free finitely generated nilpotent
  groups (or $\cal{T}$-groups, for short) obtained from unipotent
  group schemes over number rings. Let $\O_K$ be the ring of integers
  of a number field $K$ and $\Lambda$ be an $\O_K$-Lie lattice of
  nilpotency class $c$ and $\O_K$-rank $h$.
By an $\O_K$-Lie lattice we mean a free and finitely generated
$\O_K$-module, together with an antisymmetric, bi-additive form
$[\,,\,]$ satisfying the Jacobi identity. If $c>2$ we assume that
$\Lambda':=[\Lambda,\Lambda]\subseteq c!\Lambda$. This ensures that
the Hausdorff series may be used to associate to $\Lambda$ a unipotent
group scheme $\G_\Lambda$; see \cite[Section~2.1.2]{SV/14}. If $c=2$,
then the group scheme $\G_\Lambda$ may be defined more directly and
without the above assumption on $\Lambda$; cf.\
\cite[Section~2.4]{SV/14}}. For every number field $L$ containing~$K$,
with ring of integers $\O_L$, the group $\G_\Lambda(\O_L)$ is a
$\cal{T}$-group of nilpotency class $c$ and Hirsch length~$h \cdot
[L:\QQ]$.  The main result of this paper is the following.
Given arithmetic functions $f(N)$ and $g(N)$, we write `$f(N)\sim
g(N)$ as $N\rightarrow\infty$' to indicate that $\lim_{N\rightarrow
  \infty}f(N)/g(N)=1$.
\begin{thmABC}\label{thmA}
 Let $\Lambda$ and $\G=\G_\Lambda$ be as above. Then there exist
 $a(\G), \delta(\G) \in \QQ_{>0}$, and $\beta(\G)\in\NN$ such that,
 for all finite extensions $L/K$, the following hold.
 \begin{enumerate}
  \item The abscissa of convergence of $\zeta_{\G(\O_L)}(s)$ is
    $a(\G)$, i.e.\ $a(\G)=\a(\G(\O_L))$.
  \item The zeta function $\zeta_{\G(\O_L)}(s)$ has meromorphic
    continuation to the half-plane $\{s \in \CC \mid
    \rea(s)>a(\G)-\delta(\G)\}$. On the line $\{s\in\CC \mid \rea(s) =
    a(\G)\}$ the continued zeta function is holomorphic except for a
    pole at $s=a(\G)$, of order~$\beta(\G)$.
 \end{enumerate}
Moreover, there exists $c(\G(\O_L))\in\RR_{>0}$ such that
\begin{equation*}
\sum_{n=1}^N\widetilde{r}_n(\G(\O_L))\sim c(\G(\O_L)) N^{a(\G)}
(\log N)^{\beta(\G)-1}~\text{as $N\to\infty$}.
\end{equation*}
\end{thmABC}

Theorem~\ref{thmA} asserts that the invariants $a(\G)$ and $\beta(\G)$
are independent of $L$, whereas the number $c(\G(\O_L))$ may vary
with~$L$. Theorem~\ref{thmA} answers \cite[Question~1.4]{SV/14}
affirmatively.

\begin{example}\label{exa:heisenberg}
  We illustrate Theorem~\ref{thmA} in the case of the Heisenberg group
  scheme ${\bf H}$ associated to the Heisenberg $\ZZ$-Lie lattice of
  strict upper-triangular $3\times 3$-matrices. Let $\O$ be the ring
  of integers of a number field $L$. The $\cal{T}$-group ${\bf H}(\O)$
  is isomorphic to the group of upper-unitriangular
  $3\!\times\!3$-matrices over $\O$, of nilpotency class $2$ and
  Hirsch length $3[L:\QQ]$.  The zeta function of ${\bf H}(\O)$ is
  \begin{equation}\label{equ:heisenberg}
    \zeta_{{\bf
        H}(\O)}(s)=\frac{\zeta_L(s-1)}{\zeta_L(s)}=\prod_{\p\in\Spec(\O)}\frac{1-|\O/\p|^{-s}}{1-|\O/\p|^{1-s}},
\end{equation}
where $\zeta_L(s)$ is the Dedekind zeta function of~$L$ and $\p$
ranges over the nonzero prime ideals of~$\O$. This is proved in
{\cite{Magid89}} for $L=\bb{Q}$, in \cite{Ezzat} for quadratic number
fields, and in \cite[Theorem~B]{SV/14} for arbitrary number fields.
The zeta function $\zeta_{{\bf H}(\O)}(s)$ has abscissa of convergence
$a({\bf H})=2$ and may be meromorphically continued to the whole
complex plane.  The continued function has no singularities on the
line $\{s\in\bb{C}\mid \rea(s)=2\}$, apart from a simple pole at
$s=2$.  Formula~\eqref{equ:heisenberg} implies that
$$\sum_{n=1}^N\widetilde{r}_n({\bf H}(\O))\sim
\frac{1}{2\zeta_L(2)}N^2\textrm{ as $N\rightarrow \infty$}.$$ We
remark that $\zeta_L(2)$ depends subtly on~$L$;
cf.\ {\cite[Theorem~1]{Zagier}}.
\end{example}

\begin{remark}
 The part of statement (2) establishing the uniqueness of the pole at
 $s=a(\G)$ on the line $\{s\in\CC \mid \rea(s) = a(\G)\}$ is owed to
 the fact that Artin $L$-functions are used to establish the
 meromorphic continuation. Whether there are other poles in the domain
 $\{s \in \CC \mid \rea(s) > a(\G) - \delta(\G)\}$ is an interesting
 open question.
\end{remark}

Theorem~\ref{thmA} yields uniform statements for families of
$\T$-groups of a specific form, viz.\ those coming from nilpotent Lie
lattices. In the following corollary we apply Theorem~\ref{thmA} and
the methods underlying its proof to finitely generated
nilpotent groups at large.

\begin{corABC}\label{cor}
 Let $G$ be a finitely generated nilpotent group.
 \begin{enumerate}
  \item The abscissa of convergence $\alpha(G)$ of $\zeta_G(s)$ is a
    rational number.
  \item The zeta function $\zeta_G(s)$ has a meromorphic continuation
    to $\{s \in \CC \mid \rea(s)>\a(G)-\delta\}$ for some
    $\delta\in\QQ_{>0}$. On the line $\{s\in\CC \mid \rea(s) =
    \a(G)\}$, the continued zeta function has a unique pole at
    $s=\a(G)$, of order~$\beta(G)$, say.
  \item There exists $c(G)\in\RR_{>0}$ such that
$$\sum_{n=1}^N\widetilde{r}_n(G)\sim c(G) N^{\a(G)} (\log
    N)^{\beta(G)-1}~\text{as $N\to\infty$}.$$
 \end{enumerate}
Moreover, the abscissa of convergence $\a(G)$ and the pole order
$\beta(G)$ of the continued function are commensurability invariants:
if $H\leq G$ is a subgroup of $G$ of finite index in $G$, then
$\alpha(G)=\alpha(H)$ and~$\beta(G) = \beta(H)$.
\end{corABC}

\begin{remark}
The February 2015 version of \cite{Hrushovski-Martin} establishes part
(1) (but not parts (2) and (3)) of Corollary~\ref{cor} as an
application of model-theoretic results on zeta functions associated to
definable equivalence relations. We discuss this approach in some
detail in Section~\ref{subsubsec:rep.zeta.T-groups}.
\end{remark}

\subsection{Methodology}
To prove Theorem~\ref{thmA}, we present $\zeta_{\G(\O_L)}(s)$ as a
suitable Euler product and control some analytic properties of this
product. The simplest such Euler product holds for general
finitely generated nilpotent groups and is indexed by the rational
primes. Let $G$ be a finitely generated nilpotent group.  Then
\begin{equation}\label{equ:coarse.euler}
  \zeta_G(s)=\prod_{\rm{$p$ prime}}\zeta_{G,p}(s),
\end{equation}
where, for a prime $p$, the Euler factor
$\zeta_{G,p}(s):=\sum_{i=0}^\infty\widetilde{r}_{p^i}(G)p^{-is}$
enumerates the twist isoclasses of irreducible complex representations
of $G$ of $p$-power dimension.  It may also be viewed as the zeta
function of the pro-$p$ completion $\widehat{G}^p$ of $G$, enumerating
\emph{continuous} irreducible complex representations of
$\widehat{G}^p$ up to twists by continuous one-dimensional such
representations. Each Euler factor is a rational function in~$p^{-s}$;
cf.~{\cite[Theorem~1.5]{Hrushovski-Martin}}.

In the situation of Theorem~\ref{thmA}, these features may be refined
in the following sense.  Let $\Lambda$ and $\G=\G_\Lambda$ be as in
the theorem. Fix a finite extension $L$ of $K$ and write $\O = \O_L$
for its ring of integers.  For a nonzero prime ideal $\p$ of $\O$ we
denote by $\O_\p$ the completion of $\O$ at~$\p$, with residue field
cardinality $q$ and residue field characteristic~$p$.  The zeta
function $\zeta_{\G(\O)}(s)$ is the Euler product
\begin{equation}\label{equ:fine.euler}
  \zeta_{\G(\O)}(s)=\prod_{\p\in\Spec(\O)}\zeta_{\G(\O_\p)}(s)
\end{equation}
ranging over all nonzero prime ideals $\p$ of $\O$, where
$\zeta_{\G(\O_\p)}(s)$ enumerates the \emph{continuous} irreducible
complex representations of the finitely generated nilpotent pro-$p$
group $\G(\O_\p)$ up to twists by continuous one-dimensional
representations; cf.\ \cite[Proposition~2.2]{SV/14}.  Clearly, the
Euler product~\eqref{equ:fine.euler}
refines~\eqref{equ:coarse.euler}. For almost all (i.e.\ all but
finitely many) $\p$, the local zeta function $\zeta_{\G(\O_\p)}(s)$ is
a rational function in $q^{-s}$ and satisfies a functional equation
upon inversion of $q$; see \cite[Theorem A]{SV/14}. Key to this
general result is a presentation of almost all of the Euler factors
$\zeta_{\G(\O_\p)}(s)$ by \emph{formulae of Denef type}, i.e.\ finite
sums of the form
\begin{equation}\label{equ:denef.type}
\sum_{i=1}^r \vert \overline{V}_i(\o/\p) \vert W_i(q,q^{-s}).
\end{equation}
Here, the terms $\vert \overline{V}_i(\o/\p) \vert$ denote the numbers
of $\o/\p \cong \mathbb{F}_q$-rational points of the reductions modulo
$\p$ of suitable algebraic varieties $V_i$ defined over~$\O_K$ and
$W_i(X,Y)\in\QQ(X,Y)$ are bivariate rational functions enumerating
integral points in polyhedral fans. Expressions such as
\eqref{equ:denef.type} always allow for a common denominator which is
a product of factors of the form $1-q^{-as-b}$, for suitable integers
$a,b$. Expressing Euler factors through formulae of the form
\eqref{equ:denef.type} shows that they are \emph{uniformly rational}
in $q^{-s}$ in the sense of \cite[Section~6]{Hrushovski-Martin}.

Our proof of Theorem~\ref{thmA} proceeds by analysing Euler products
of formulae of Denef type arising in the specific context of
representation zeta functions of finitely generated nilpotent groups.
Here, such formulae are facilitated by the Kirillov orbit method,
which allows for a parametrization of twist isoclasses of
finite-dimensional representations by co-adjoint orbits.  Roughly
speaking, we show that the Euler products of formulae of Denef type
representing almost all factors in the Euler
product~\eqref{equ:fine.euler} can be approximated by products of
translates of the Artin $L$-functions to continue the relevant Euler
product to the left of its abscissa of convergence. This analysis is
motivated by and uses tools from \cite{duSG/00}, but extra care has to
be taken to adapt this analysis to the specific kind of Denef formulae
arising and to render it uniform over all extensions $L$ of~$K$. We
adapt methods from \cite{Voll-compact-p-adic-analytic-arithmetic} to
show that the finitely many Euler factors not covered by our generic
arguments do not affect the relevant analytic properties of the
cofinite Euler product. A Tauberian theorem allows for the asymptotic
final statement.

The methods we employ are inspired by but different from both those
used to study the analytic properties of Euler products of $\p$-adic
\emph{cone integrals} arising in the theory of subgroup growth of
$\T$-groups and those used to approximate Euler products arising in
the representation growth of arithmetic subgroups of semisimple
groups. We discuss these classes of $\p$-adic integrals in some detail
in Sections~\ref{subsec:subgroup} and~\ref{subsec:semisimple}.

Corollary~\ref{cor}, which pertains to general finitely generated
nilpotent groups, is obtained by applying Theorem~\ref{thmA} to a
subgroup of finite index which fits into its remit. The existence of
such subgroups follows from the Mal'cev correspondence.

\subsection{Background and related research} 

\subsubsection{Representation zeta functions of nilpotent groups}\label{subsubsec:rep.zeta.T-groups}
The conclusions of Theorem~\ref{thmA} and Corollary~\ref{cor} were
previously observed in special cases.  In \cite[Theorem~B]{SV/14}, for
instance, three infinite families of group schemes generalizing the
Heisenberg group scheme (see Example~\ref{exa:heisenberg}) were
studied. The explicit formulae given there for the zeta functions of
the members of these families in terms of Dedekind zeta functions show
that the respective abscissae of convergence are integers and
independent of the number field. The respective constants in the
asymptotic expansions, however, are of an arithmetic nature, being
expressible in terms of special values of Dedekind zeta functions; see
\cite[Corollary~1.3]{SV/14}. Further examples of representation zeta
functions of $\T$-groups, illustrating Theorem~\ref{thmA} and
Corollary~\ref{cor}, were computed in \cite{Ezzat/12, Snocken}. It is
not hard to show that every positive rational number is the abscissa
of convergence of some $\cal{T}$-group of nilpotency class~$2$; see
\cite[Theorem~4.22]{Snocken}. It is an interesting open problem to
characterize analytic data such as the invariants $a(\G)$,
$\delta(\G)$, and $\beta(\G)$ in Theorem~\ref{thmA} in terms of
structural invariants of the group scheme. Note that
Theorem~\ref{thmA} makes no claim as to the maximal value of
$\delta(\G)$. In the examples discussed, $\delta(\G)$ can be chosen
arbitrarily large. We expect it to be easy, however, to exhibit
examples of $\T$-groups whose representation zeta functions do not
have meromorphic continuation to the whole plane, but where this
continuation has a natural boundary.

We already mentioned the fundamental contributions made in
\cite{Hrushovski-Martin}. This paper -- which developed substantially
between its first posting in 2006 and its latest revision in 2015,
acquiring two additional authors on the way -- studies the more
general setup of zeta functions enumerating classes of definable
equivalence relations, a model-theoretic notion. This class of zeta
functions contains representation zeta functions of finitely generated
nilpotent groups. Already in its first version, by Hrushovski and
Martin, \cite{Hrushovski-Martin} established the rationality in
$p^{-s}$ of \emph{all} factors in the coarse Euler
product~\eqref{equ:coarse.euler}. The 2015-version established
formulae for \emph{almost all} of the fine Euler factors in
\eqref{equ:fine.euler} which are very similar to
\eqref{equ:denef.type} (cf.\ \cite[eq.~(8.6)]{Hrushovski-Martin}) and
proved the rationality of the abscissa of convergence of the
representation zeta function of an arbitrary finitely generated
nilpotent group as part of \cite[Theorem~1.5]{Hrushovski-Martin}. (It
appears that the methods of \cite{Hrushovski-Martin} might be
sufficient to establish rationality of \emph{all} Euler factors in
\eqref{equ:fine.euler}, though this is not made explicit in
\cite{Hrushovski-Martin}.) The paper neither implies the other
statements of Corollary~\ref{cor} nor the uniformity statements of
Theorem~\ref{thmA}. In fact, (uniform) statements about meromorphic
continuation are not expected to hold in the general setting of zeta
functions associated to definable equivalence relations. Rather, their
validity seems to reflect specific features of representation growth
of finitely generated nilpotent groups; cf.\ also
\cite[Remark~6.6]{Hrushovski-Martin}.

In \cite{Rossmann/15}, Rossmann introduces and studies so-called
topological representation zeta functions associated to unipotent
algebraic group schemes such as the ones considered in the present
paper. Informally, these rational functions are obtained as limits of
local representation zeta functions as `$p \rightarrow 1$'; see
\cite{Rossmann/15} for further details and results.

\subsubsection{Representation growth of `semisimple' arithmetic
  groups}\label{subsec:semisimple} The theory of representation growth
of groups is arguably most developed for arithmetic subgroups of
semisimple algebraic groups over number fields. To be more precise,
let $\G$ be a semisimple group scheme defined over a number field $K$,
together with a fixed embedding $\G \hookrightarrow \GL_d$ for some
$d\in\NN$. Consider, for simplicity's sake, the arithmetic group
$\Gamma = \G(\O_S)$, where $\O_S$ the ring of $S$-integers $\O_S$ of
$K$ for some finite set of places $S$ of $K$, including all the
archimedean ones. If $\Gamma$ satisfies the (strong) Congruence
Subgroup Property (CSP), then the sequence $(r_n(\Gamma))_{n\in\NN}$,
enumerating the $n$-dimensional irreducible complex representations of
$\Gamma$, grows at most polynomially (\cite{LubotzkyMartin/04}) and
the zeta function $\zeta_{\Gamma}(s) := \sum_{n=1}^\infty
r_n(\Gamma)n^{-s}$ is an Euler product of the form
$$\zeta_{\Gamma}(s) = \zeta_{\G(\CC)}(s)^{\vert K : \QQ \vert}
\prod_{v\not\in S}\zeta_{\G(\O_v)}(s);$$
cf.\ \cite[Proposition~1.3]{LarsenLubotzky/08}. Here,
$\zeta_{\G(\CC)}(s)$ denotes the \emph{Witten zeta function},
enumerating the finite-dimensional rational representations of the
algebraic group $\G(\CC)$, whereas each non-archimedean factor
$\zeta_{\G(\O_v)}(s)$ enumerates the continuous representations of the
compact $p$-adic analytic group $\G(\O_v)$. Here, $\O_v$ denotes the
completion of $\O$ at the place~$v$ of~$K$. The Witten zeta functions
are comparatively well understood; see, for instance,
\cite[Theorem~5.1]{LarsenLubotzky/08} for their abscissa of
convergence. To control the non-archimedean factors and the analytic
properties of their Euler product seems much harder. The factors are
not, in general, rational functions in~$p^{-s}$, as the groups
$\G(\O_v)$ are only virtually pro-$p$. Consequently, formulae of Denef
type (cf.~\eqref{equ:denef.type}) cannot be expected for these Euler
factors in general.  A result by Jaikin (\cite{Jaikin-zeta}) expresses
the zeta functions of the groups $\G(\O_v)$ as rational functions in
$p^{-s}$ and terms of the form $n_i^{-s}$ for finitely many integers
$n_i$, but to what degree these expressions may be described uniformly
as the numbers of rational points of globally defined geometric
objects (such as varieties or, more generally, definable sets)
remains unclear.

Most general results about the analytic properties of Euler products
of the non-archi\-medean factors are based on suitable approximations
of the individual Euler factors. Using such techniques, it was shown
that the abscissa of convergence of $\zeta_\Gamma(s)$ is a rational
number (\cite{Avni-Annals}) which only depends on the absolute root
system of the algebraic group $\G$; cf.~\cite{Voll-base-extension}.
For groups of type $A_2$ this abscissa is equal to $1$, the zeta
function allows for meromorphic continuation to at least $\{s\in\CC
\mid \rea(s) > 5/6\}$, and the continued function is analytic in this
domain with the exception of a double pole at $s=1$; cf.\
\cite{AKOVII/14}. It is not known whether this kind of uniform
analytic properties are a general feature of the arithmetic groups
considered in this context.  
In \cite{AizenbudAvni/15}, Aizenbud and Avni establish absolute upper
bounds for the abscissae of convergence of representation zeta
functions of arithmetic lattices of higher rank. We refer to
\cite{Klopsch/13} and \cite{Voll-survey-St.Andrews} for surveys of
representation zeta functions of groups.

\subsubsection{(Normal) subgroup growth of nilpotent
  groups}\label{subsec:subgroup} One precursor of and source of
inspiration in the study of zeta functions in representation growth is
the use of these functions in the related area of subgroup growth of
groups.  The idea to employ zeta functions to enumerate finite-index
subgroups of finitely generated nilpotent groups was pioneered in
\cite{GSS88} by Grunewald, Segal, and Smith and subsequently developed
by many authors; see \cite{SubgroupGrowth, Voll-new-comer-guide} and
references therein.  Zeta functions enumerating finite-index (normal)
subgroups of finitely generated nilpotent groups are Dirichlet
generating functions which satisfy decompositions akin
to~\eqref{equ:coarse.euler} as Euler products of rational functions,
indexed by the rational primes. In~\cite{duSG/00}, du Sautoy and
Grunewald produced formulae of Denef type for the Euler factors
arising in this context, by expressing almost all of them as $p$-adic
\emph{cone integrals}. One of the main results of their paper
establishes analytic properties of Euler products of cone integrals;
cf.\ \cite[Theorem~1.5]{duSG/00}. Together with some control over
remaining Euler factors, this allows them to establish
\cite[Theorem~5.7]{duSG/00}, an analogue of our Corollary~\ref{cor}
for (normal) subgroup zeta functions of groups.

The analogy notwithstanding, Corollary~\ref{cor} does not seem to
follow directly from the results in \cite{duSG/00}, as it is not known
whether (almost all of) the Euler factors of representation zeta
functions of finitely generated nilpotent groups are cone
integrals. This is one of the reasons why novel $p$-adic methods for
the description of the relevant local zeta functions were introduced
in \cite{Hrushovski-Martin}
and~\cite{Voll-Functional-equations-annals}. The latter paper produced
explicit formulae of Denef type for almost all Euler factors
in~\eqref{equ:coarse.euler}. The ideas behind this analysis were
refined in \cite{SV/14} to produce such formulae for (almost all of)
the factors of the ``fine'' Euler products~\eqref{equ:fine.euler}. In
the current paper, we adapt ideas pioneered in \cite{duSG/00} for the
study of Euler products of such formulae, regardless of the question
-- which we cannot decide -- whether or not they are cone integrals.

Having commented on some similarities between subgroup growth and
representation growth of finitely generated nilpotent groups, we
mention a striking dissimilarity: uniform results such as
Theorem~\ref{thmA}, establishing the independence of key analytic data
of zeta functions under base extension, do not hold for normal zeta
functions of finitely generated nilpotent groups. This is already
witnessed by the zeta function $\zeta_{{\bf
    H}(\O)}^{\triangleleft}(s)$ enumerating finite-index normal
subgroups of the Heisenberg groups ${\bf H}(\O)$ over~$\O$: it
satisfies a ``coarse'' Euler product similar to
\eqref{equ:coarse.euler}, but not a ``fine'' one indexed by the places
of $L$, such as~\eqref{equ:fine.euler}. The (coarse) Euler factors
depend in a combinatorially intricate manner on the indexing prime's
decomposition in $L$; cf.\ \cite[Theorem~3]{GSS88}. The abscissa of
convergence of $\zeta_{{\bf H}(\O)}^{\triangleleft}(s)$ is
$2[L:\bb{Q}]$; see \cite{Voll-Schein-I, Voll-Schein-II} for further
details.

\subsection{Organization and notation}
In Section~\ref{sec:rep.T.groups}, we recall from \cite{SV/14} some
tools for the analysis of local representation zeta functions arising
as factors in products such as~\eqref{equ:fine.euler} in terms of
Poincar\'e series associated to matrices of linear forms and $\p$-adic
integrals.  In Section~\ref{sec:denef}, we combine results of and
techniques developed in \cite{duSG/00}, \cite{SV/14}, and
\cite{Voll-compact-p-adic-analytic-arithmetic} to represent these
local zeta functions by formulae of Denef type.  In particular, we
prove that the abscissae of convergence of the local zeta functions
are all elements of a finite set of rational numbers (cf.\ Corollary
\ref{local poles}) which is strictly dominated by the (global)
abscissa of convergence.  Theorem~\ref{thmA} and Corollary~\ref{cor}
are proved in Sections~\ref{sec:proof.thmA} and~\ref{sec:proof.thmB},
respectively.

Given a ring $A$, we write $\Mat_{n\times m}(A)$ for the set of all
$n\times m$ matrices with coefficients in $A$ and $\Mat_{n}(A)$
for~$\Mat_{n\times n}(A)$.  We write $\NN$ for the set of positive
integers and $\NN_0=\NN\cup\{0\}$. Given $k\in\NN_0$, we set $[k] =
\{1,\cdots,k\}$ and $[k]_0=[k]\cup\{0\}$.  Given $r\in\RR$, we write
$\lfloor r\rfloor$ for the largest integer not exceeding $r$.  Given a
variable $x$ and a set of polynomials $F(y)$, we write $xF(y)$ for the
set $\{xf(y) \mid f(y)\in F(y)\}$.

\subsection*{Acknowledgments}
We acknowledge support from the DFG Sonderforschungsbe\-reich 701 at
Bielefeld University. We thank Sebastian Herr, Benjamin Martin and
Tobias Rossmann for several helpful discussions and an anonymous
referee for very valuable feedback. Voll thanks the University of
Auckland and the Alexander von Humboldt Foundation for support during
the final phase of work on this paper.

\section{Representation zeta functions of finitely generated nilpotent
  groups}\label{sec:rep.T.groups}
Let $K$ be a number field and $\O_K$ its ring of integers.  Let
$\Lambda$ be a nilpotent $\O_K$-Lie lattice of nilpotency class
$c$. Without loss of generality~$c\geq 2$.  If $c>2$ we assume that
$\Lambda'\subseteq c!\Lambda$. Let $\G=\G_\Lambda$ be the unipotent
group scheme over $\O_K$ associated to $\Lambda$ as in
{\cite[Sections~2.1.2 and~2.4.1]{SV/14}}. Let $L$ be a finite
extension of $K$ and denote by $\O=\O_L$ its ring of integers. For a
nonzero prime ideal $\p$ of $\O$, we denote by $\O_\p$ the completion
of $\O$ at $\p$, with maximal ideal $\p$, residue field
cardinality~$q$, and residue field characteristic $p$.  Recall that
the Euler product~\eqref{equ:fine.euler} relates the zeta function of
the $\T$-group $\G(\O)$ to the zeta functions of the nilpotent pro-$p$
groups~$\G(\O_\p)$, enumerating the continuous twist isoclasses of
continuous irreducible complex representations of~$\G(\O_\p)$.

In this section, we recapitulate some notation and constructions from
\cite{SV/14} to express almost all of the Euler factors
$\zeta_{\G(\O_\p)}(s)$ in~\eqref{equ:fine.euler} in terms of
Poincar\'e series associated to matrices of linear forms and these
series, in turn, in terms of $\p$-adic integrals. The key tool for
this approach is the Kirillov orbit method for pro-$p$ groups such
as~$\G(\O_\p)$.  When it is applicable, this method allows us to
translate the problem of counting irreducible representations into the
problem of counting co-adjoint orbits of an associated Lie algebra. It
turns out that, for all but finitely many prime ideals~$\p$, the
Kirillov orbit method applies to $\G(\O_\p)$; in the remaining cases
it applies to suitable finite index subgroups of~$\G(\O_\p)$.

In Section~\ref{sec:denef} we will use the results of the current
section to develop formulae of Denef type which uniformly describe
almost all of the Euler factors of the
product~\eqref{equ:fine.euler}. They will allow us, in particular, to
define a finite set $\mathbf{P}$ of rational numbers, depending only
on $\G$, containing the abscissae of convergence of \emph{all} the
Euler factors~$\zeta_{\G(\O_\p)}(s)$. (Here, the abscissa of
convergence of the zeta function of a finitely generated nilpotent
pro-$p$ group is defined as for abstract finitely generated nilpotent
groups.)  In this context, we will refer to the following elementary
lemma which implies that the abscissa of convergence of the zeta
function of a finitely generated nilpotent pro-$p$ group is a
commensurability invariant. Its proof is analogous to that of
{\cite[Proposition 4.13]{Snocken}}, which in turn is a modification of
that of {\cite[Lemma 2.2]{LubotzkyMartin/04}}.

\begin{lemma}\label{abscissa commensurable}
Let $p$ be a prime, $G$ a finitely generated nilpotent pro-$p$ group,
and $H \leq G$ a subgroup of finite index in $G$.  Write $|G:H|=p^a$
and $|G'\cap H:H'|=p^b$ for $a,b\in\NN_0$.  For $n\in\NN_0$, let
$\R_{p^n}(G)=\sum_{i=0}^n\widetilde{r}_{p^i}(G)$ and
$\R_{p^n}(H)=\sum_{i=0}^n\widetilde{r}_{p^i}(H)$ denote the numbers of
continuous twist isoclasses of complex irreducible representations of
dimension at most $p^n$ of $G$ and $H$, respectively.  Then
\begin{equation*}
\R_{p^n}(H)\leq p^a\R_{p^{a+n}}(G)\quad\textrm{ and }\quad
\R_{p^n}(G)\leq p^{a+b}\R_{p^n}(H).
\end{equation*}
In particular, the abscissa of convergence of the representation zeta
function $\zeta_G(s)$ is a commensurability invariant of~$G$,
i.e.\ $\a(G)=\a(H)$.
\end{lemma}

\begin{proof}
For $i\in\NN_0$, let $\Irr_{p^i}(G)$ resp.\ $\Irr_{p^i}(H)$ be the set
of twist isoclasses of $G$ resp.\ $H$ of dimension~$p^i$. (Here and in
the sequel, the dimension of a twist isoclass denotes, of course, the
common dimension of all of its elements.) For $i\in[n]_0$ and
$\widetilde{\sigma}\in\Irr_{p^i}(H)$, let
$\widetilde{\Phi}(\widetilde{\sigma})$ be the twist isoclass of an
irreducible constituent of $\Ind^G_H\sigma$, where $\sigma$ is any
representative of the twist isoclass~$\widetilde{\sigma}$. This
defines a (noncanonical) map
$$\widetilde{\Phi}:\bigcup_{i=0}^n\Irr_{p^i}(H) \rightarrow
\bigcup_{i=0}^{n+a}\Irr_{p^i}(G).$$ To establish the first inequality
it suffices to show that the fibre sizes of $\widetilde{\Phi}$ are all
bounded by~$p^a$. Fix thus $\widetilde{\rho}\in\bigcup_{i =
  0}^{n+a}\Irr_{p^i}(G)$ and consider
$\widetilde{\Phi}^{-1}(\widetilde{\rho}) = \{
\widetilde{\sigma}_1,\dots,\widetilde{\sigma}_m\}\subseteq
\bigcup_{i=0}^n \Irr_{p^i}(H)$. Without loss of generality we may
assume that $\dim \widetilde{\sigma}_1 \leq \dim \widetilde{\sigma}_i$
for all $i\in\{2,\dots,m\}$.

By definition of $\widetilde{\Phi}$, there exist $\sigma_i \in
\widetilde{\sigma}_i$ and $\rho_i\in\widetilde{\rho}$ such that
$\rho_i\in\Ind^G_H\sigma_i$ for each~$i$. By Frobenius Reciprocity,
$\sigma_i\in\Res^G_H \rho_i$ for each~$i$. There exist thus continuous
$1$-dimensional representations $\chi_i$ of $G$ such that $(\Res^G_H
\chi_i) \otimes \sigma_i \in \Res^G_H\rho_1$. This implies that
$\Res_H^G\rho_1=(\Res_H^G\chi_1)\otimes\sigma_1\oplus\cdots\oplus(\Res_H^G\chi_m)\otimes
\sigma_m\oplus\rho'$ for some continuous representation $\rho'$ of
$H$.  Hence $m\dim\sigma_1\leq\dim\rho_1 \leq p^a\dim\sigma_1$ and
thus $m\leq p^a$.

The proof of the second inequality is similar.
\end{proof}

\subsection{Poincar\'e series}\label{subsec:poincare}
We now fix a nonzero prime ideal $\p$ of $\O$ and write $\o=\O_\p$.
Let $\g:=\Lambda(\o)=(\Lambda\otimes_{\O_K}\O)\otimes_\O\o $ and write
$\frak{z}$ for the centre of $\g$. Set $h=\rm{rk}_\o(\g)$ and
$$ d=\rm{rk}_\o(\g'), ~ k=\rm{rk}_\o(\i(\g')/\i(\g'\cap\frak{z}))
=\rm{rk}_\o(\i(\g'+\frak{z})/\frak{z}), ~
r-k=\rm{rk}_\o(\g/\i(\g'+\frak{z})),
$$ so that $r=\rm{rk}_\o(\g/\frak{z})$. Here, for a finitely generated
$R$-module $M$ with $R$ either $\O$ or $\o$ and an $R$-submodule $N$
of~$M$, we denote by $\i(N)$ the isolator of $N$ in $M$, that is the
smallest submodule $L$ of $M$ containing $N$ such that $M/L$ is
torsion-free. Note that $\frak{z}$ is isolated in $\g$,
i.e.\ $\i(\frak{z})=\frak{z}$; see \cite[Lemma~2.5]{SV/14}.  Notice
also that $c\leq 2$ if and only if~$k=0$.

We choose a uniformizer $\pi$ of $\o$, write $^{-}$ for the natural
surjection $\g\to\g/\frak{z}$, and choose an ($\o$-)basis ${\bf
  e}=(e_1,\cdots ,e_h)$ for $\g$ such that 
\begin{align*}
(e_{r-k+1},\cdots,e_r) & \textrm{ is a basis for
   $\i(\overline{\g'+\frak{z}})$},\\
(e_{r+1},\cdots, e_{r-k+d}) & \textrm{ is a basis for
   $\i(\g'\cap\frak{z})$, and}\\
(e_{r+1},\cdots,e_h) & \textrm{ is a basis for~$\frak{z}$.}
\end{align*}
It follows from the elementary divisor theorem that there exist
$b_1,\cdots,b_d\in\NN_0$ such that
$(\overline{\pi^{b_1}e_{r-k+1}},\cdots,\overline{\pi^{b_k}e_r})$ and
$(\pi^{b_{k+1}}e_{r+1},\cdots, \pi^{b_d}e_{r-k+d})$ are bases for
$\overline{\g'+\frak{z}}$ and $\g'\cap\frak{z}$ respectively. We
choose a basis ${\bf f}=(f_1,\cdots,f_d)$ for $\g'$ such that
\begin{align*}
(\overline{f_1},\cdots,\overline{f_k}) & = (\overline{\pi^{b_1}e_{r-k+1}},\cdots,\overline{\pi^{b_k}e_r}) \textrm{ and }\\
(f_{k+1},\cdots,f_{d}) & = (\pi^{b_{k+1}}e_{r+1},\cdots,
   \pi^{b_d}e_{r-k+d});
\end{align*}
 cf.\ {\cite[Section~2.2.2]{SV/14}}.

For $i,j\in [r]$ and $l\in[d]$, let $\lambda_{ij}^l\in\o$ be the
structure constants with respect to above bases,
i.e.\ $[e_i,e_j]=\sum_{l=1}^d\lambda_{ij}^lf_l$. We define the
commutator matrix
$$\cal{R}(\bfY)=\left(\sum_{l=1}^d\lambda_{ij}^lY_l \right)_{ij}
\in\Mat_r(\o[\bfY])$$ of $\g$ with respect to the chosen bases and
let $\cal{S}(\bfY)$ be the $r\times k$ submatrix of $\cal{R}(\bfY)$
comprising the last $k$ columns of $\cal{R}(\bfY)$.  If $c=2$, then
the matrix $\cal{S}$ does not feature since $k=0$.

Let $m\in\NN_0$ with $m\leq r$. We say that a matrix
 $S\in\Mat_{r\times m}(\o)$ has (elementary divisor) type
 $\widetilde{\nu}(S)={\bf c}=(c_1,\cdots,c_m) \in
 (\NN_0\cup\{\infty\})^m$ if $S$ is equivalent to 
 $$\begin{pmatrix}
 \pi^{c_1}& & \\
 & \ddots & \\
 & & \pi^{c_m} \\
 & & \\
 \end{pmatrix}\in\Mat_{r\times m}(\o),$$
 with $0\leq c_1\leq\cdots\leq c_m$. Let $R\in\Mat_r(\o)$ be an
 antisymmetric matrix. We set $\nu(R)= (a_1,\cdots,a_{\lfloor
   r/2\rfloor})\in (\NN_0\cup\{\infty\})^{\lfloor r/2 \rfloor}$, where
 $0\leq a_1\leq\cdots\leq a_{\lfloor r/2 \rfloor}$, if
 $$
 \widetilde{\nu}(R)=\left\{
 \begin{array}{ll}
 (a_1,a_1,a_2,a_2,\cdots,a_{r/2,r/2}) & \rm{if $r$ is even,} \\
 (a_1,a_1,a_2,a_2,\cdots,a_{(r-1)/2,(r-1)/2},\infty) & \rm{if $r$ is odd.}
 \end{array}
 \right.$$

 Let $N\in\NN_0$. Given an antisymmetric matrix
 $\overline{R}\in\Mat_r(\o/\p^N)$, we set
 $\nu(\overline{R}):=(\min\{a_i,N\})_{i\in[\lfloor
     r/2\rfloor]}\in([N]_0)^{\lfloor r/2\rfloor}$, where ${\bf a} =
 (a_1,\dots,a_{\lfloor r/2 \rfloor}) =\nu(R)$ is the type of any lift
 $R$ of $\overline{R}$ under the natural surjection
 $\Mat_r(\o)\to\Mat_r(\o/\p^N)$.  Given $\overline{S}\in\Mat_{r\times
   k}(\o/\p^N)$, the vector $\widetilde{\nu}(\overline{S})\in
 ([N]_0)^k$ is defined similarly. We set $W_N(\o):=(\o/\p^N)^d
 \setminus (\p/\p^N)^d$ if $N\in\NN$ and $W_0(\o)=\{0\}$.  Given
 $N\in\NN_0$, ${\bf a}\in\NN_0^{\lfloor r/2\rfloor}$, and ${\bf
   c}\in\NN_0^k$, we set
 $$\cal{N}^\o_{N,{\bf a},{\bf c}}:=\#\{ {\bf y}\in
 W_N(\o) \mid\nu(\cal{R}({\bf y}))={\bf a}, \,
 \widetilde{\nu}(\cal{S}({\bf y})\cdot
 \diag(\pi^{b_1},\cdots,\pi^{b_k}))={\bf c}\},$$ giving rise to the
 Poincar\'{e} series
 \begin{equation}\label{Poincare define}
 \cal{P}_{\cal{R},\cal{S},\o}(s):= \sum_{\substack{N\in\NN_0 \\ {\bf
       a}\in\NN_0^{\lfloor r/2\rfloor}, \,{\bf c}\in\NN_0^k}}
 \cal{N}^\o_{N,{\bf a},{\bf c}}q^{ -\sum_{i=1}^{\lfloor
     r/2\rfloor}(N-a_i)s -\sum_{i=1}^k(N-c_i) }.
 \end{equation}

 \begin{proposition} {\cite[Propositions~2.9 and 2.18]{SV/14}}
 \label{pro:poin.p}
 If $(p,c)\neq(2,3)$, then
 $$\zeta_{\G(\O_\p)}(s)=\cal{P}_{\cal{R},\cal{S},\O_\p}(s).$$
 \end{proposition}
Key to writing the zeta function $\zeta_{\G(\O_\p)}(s)$ as a
Poincar\'{e} series is the fact that the hypothesis of
Proposition~\ref{pro:poin.p} ensures that the Kirillov orbit method is
applicable to the pro-$p$ group~$\G(\O_\p)$; see
{\cite[Section~2.2]{SV/14}} for details.

Assume now that $p=2$ and $c=3$.  In this case we consider, instead of
$\g$, a suitable congruence sublattice of $\g$ giving rise to a
pro-$2$ group to which the Kirillov orbit method is applicable. Let
$e=e(\o,\ZZ_2)$ be the absolute ramification index of~$\o$,
i.e.\ $2\o=\p^e\o$, and consider $\g^{e}=2\Lambda\otimes_\O\o$. Since
$[2\Lambda,2\Lambda]=4\Lambda'\subseteq 4\cdot 3!\Lambda \subseteq
4\cdot 2\Lambda$, we have that $(\g^{e})'\subseteq 4\g^{e}$,
i.e.\ $\g^{e}$ is powerful. Hence the group $\G^e(\o) := \exp(\g^e)$
is a powerful pro-$2$ group.  Moreover, since $\G^e(\o)$ is finitely
generated torsion-free, {\cite[Theorem~4.5]{Segal-analytic-pro-p}}
yields that the group $\G^e(\o)$ is a uniform pro-$2$ group. By
{\cite[Theorem~2.12]{Jaikin-zeta}}, there exists a Kirillov
correspondence between the finite co-adjoint orbits in the dual of the
Lie algebra $\g^{e}$ and the continuous irreducible representations of
$\G^{e}(\o)$. Multiplying the chosen basis ${\bf e}$ for $\g$ by
$\pi^e$ yields a basis $\pi^e{\bf e}$ for $\g^{e}$.  Hence, the zeta
function $\zeta_{\G^e(\o)}(s)$ can be expressed in term of a
Poincar\'e series as in \eqref{Poincare define}.

\subsection{$\p$-Adic integration}\label{subsec:int}
 It follows from
 {\cite[Section~2.2]{Voll-Functional-equations-annals}} that
 \begin{equation}\label{poincare-I=1}
  \cal{P}_{\cal{R},\cal{S},\frak{o}}(s)=1+(1-q^{-1})^{-1}\cal{Z}_\frak{o}(-s/2,-1,us+v-d-1)),
 \end{equation}
 where $\cal{Z}_\frak{o}(\rho,\sigma,\tau)$ is the $\p$-adic integral
 defined in {\cite[(2.8)]{SV/14}} as follows:
 \begin{multline}\label{Zpadic}
  \cal{Z}_\frak{o}(\rho,\sigma,\tau)=\\\Zpadic.
 \end{multline}
 Here $\mu$ is the additive Haar measure on $\o^{d+1}$ is normalised so
 that $\mu(\o^{d+1})=1$, and
 \begin{align}
   2u & =  \max\{\rm{rk}_{\rm{Frac}(\o)}(\cal{R}({\bf z}))\mid{\bf z}\in\o^d \},\nonumber\\
   v & =  \max\{\rm{rk}_{\rm{Frac}(\o)}(\cal{S}({\bf z}))\mid {\bf z}\in\o^d \},\nonumber\\
   {F}_j(\bfY) & = \{ f\mid f=f(\bfY)~\rm{a principal}~2j\times
   2j ~\rm{minor of}
   ~\cal{R}(\bfY) \},\label{eq:poly F}\\
   {G}_i(\bfY) & = \{g\mid g=g(\bfY)~\rm{an $i\times i$ minor
     of}~
   \cal{S}(\bfY)\cdot\diag(\pi^{b_1},\cdots,\pi^{b_k})\},\label{eq:poly G}\\
   \|H(X,\bfY)\|_\p & = \max\{|h(X,\bfY)|_\p\mid h\in H\}
   ~\rm{for a finite set}~H\subset\o[X,\bfY].\nonumber
 \end{align} 

Notice that $v=0$ if $c=2$. 

\section{Formulae of Denef type for local representation zeta
  functions}\label{sec:denef}
 In this section we produce uniform formulae of Denef type for (almost
 all of the) local zeta functions $\zeta_{\G(\O_\p)}(s)$ arising as
 Euler factors in~\eqref{equ:fine.euler}.  They enable us to obtain
 information about the abscissae of convergence of all Euler factors
 (Corollary~\ref{local poles}) and will be put to use in
 Section~\ref{sec:proof.thmA}, where Theorem~\ref{thmA} will be
 proven.

 For a nonzero prime ideal $\p$ of $\O$, we write $\o=\O_\p$ for the
 completion of $\O$ at~$\p$.  As we have seen in
 Section~\ref{sec:rep.T.groups}, almost all of the factors
 $\zeta_{\G(\o)}(s)$ in the Euler product \eqref{equ:fine.euler} are
 given by Poincar\'{e} series $\cal{P}_{\cal{R},\cal{S},\frak{o}}(s)$,
 which can be expressed in terms of the $p$-adic integral
 $\cal{Z}_\frak{o}(\rho,\sigma,\tau)$ defined in~\eqref{Zpadic}. In
 order to analyze such integrals in a uniform manner we recall from
 {\cite[Section~2]{Voll-Functional-equations-annals}} some formulae
 for $\p$-adic integrals of the following general form. Fix
 $d,l\in\NN$. Let $I_\k, \k\in[l]$, be finite index sets, $x$ and
 ${\bf y}=(y_{ij})_{1\leq i,j \leq d}$ integration variables, and fix
 nonnegative integers $e_{\k\i}$ and finite sets of polynomials
 $F_{\k\i}(\bfY)$ over~$\O_K$ for $\k\in[l]$, $\iota\in I_\kappa$.  We
 define
 \begin{equation}\label{equ:general.Z}
   Z({\bf s})= Z(s_1,\dots,s_l):=\ZIpadic.
 \end{equation} 
 Here ${\bf s}=(s_1,\cdots,s_l)$ is a vector of complex variables and
 $d\mu$ denotes the product of the additive Haar measure on $\p$ and
 the Haar measure on $\GL_d(\o)$, normalized so that $d\mu(\p \times
 \GL_d(\o)) = q^{-1}\prod_{i=1}^d(1-q^{-i})$. We assume throughout
 that the ideals $(F_{\kappa\iota})$ are all invariant under the
 natural (right-)action of the standard Borel subgroup $B \subseteq
 \GL_d$. Equation~\eqref{equ:general.Z} is a slight specialization of
 the integral given in
 \cite[eq.~(6)]{Voll-Functional-equations-annals}: we make some of the
 hypotheses of \cite[Theorem~2.2]{Voll-Functional-equations-annals}
 and restrict to the case~$I=\{1\}$. Note that whilst
 \eqref{equ:general.Z} is a $\p$-adic (``local'') integral, its
 integrand is defined over~$\O_K$, i.e.\ ``globally''.

 We fix a principalization $(Y,h)$ with $h:Y\rightarrow \GL_d/B$ of
 the $\O_K$-ideal $$\cal{I}=\prod_{\k\in[l]}\prod_{\i\in
   I_\k}(F_{\k\i}).$$ Let $(N_{u\k\i},\nu_u)_{u\k\i}$ be the
 associated numerical data, with $u\in T$, $\k\in[l]$, and $\i\in
 I_\k$. Here, $T$ is a finite set indexing the irreducible components
 $E_u$ of the pre-image under $h$ of the variety defined
 by~$\cal{I}$. For $u,\k, \i$ as above, the number $\nu_u-1$ denotes
 the multiplicity of $E_u$ in the divisor $h^*(d\mu(\bfy))$ and
 $N_{u\k\i}$ denotes the multiplicity of $E_u$ in the pre-image under
 $h$ of the variety defined by the ideal $(F_{\k\i})$.  The
 principalization $(Y,h)$ has good reduction modulo $\p$ (cf.\
 {\cite[Definition 2.2]{Denef-degree}}) for all but finitely
 many~$\p$; cf.~\cite[Theorem 2.4]{Denef-degree}. In the remaining
 cases we say that $(Y,h)$ has bad reduction modulo~$\p$.

 In Sections \ref{subsec:good.red} and \ref{subsec:bad.red} we compute
 formulae of Denef type for $Z({\bf s})$ in case that $(Y,h)$ has good
 respectively bad reduction modulo $\p$. In Section~\ref{subsec:app}
 we specialize these formulae to the
 integrals~$\cal{Z}_\frak{o}(\rho,\sigma,\tau)$ defined in
 Section~\ref{subsec:int}.
 \subsection{Good reduction}\label{subsec:good.red}
 Assume that $\p$ is such that $(Y,h)$ has good reduction modulo~$\p$.
 According to {\cite[Theorem 2.2]{Voll-Functional-equations-annals}},
 \begin{equation}\label{equ:Z(s).general.good}
 Z({\bf s})= \frac{ (1-q^{-1})^{d+1}}{ q^{\binom{d}{2}} }
 \sum_{U\subseteq T} c_U(\o/\p)(q-1)^{|U|}\Xi_{U}(q,{\bf s}).
 \end{equation}
 Here, for each $U\subseteq T$, the coefficient $c_U(\o/\p)$ is the
 number of $\o/\p$-rational points of
 $\overline{E_U}\setminus\cup_{V\varsupsetneq U}\overline{E_V}$
 $(E_U:=\cap_{u\in U}E_u)$, where $~\bar{}~$ denotes reduction modulo
 $\p$.  Furthermore, $q$ denotes the cardinality of the residue field
 $\o/\p$ and
 \begin{equation}\label{def:Xi_U}
 \Xi_{U }(q,{\bf s}) = \sum_{ \substack{(m_u)_{u\in U}\in\NN^{|U|}
     \\ m_{t+1}\in\NN} } q^{-\mathcal{L}((m_u)_{u\in U},m_{t+1})
   -\sum_{\k\in[l]}s_\k\min\{\mathcal{L}_{\k\i}((m_u)_{u\in
     U},m_{t+1})\mid\iota\in I_{\kappa}\} },
 \end{equation}
 where $t=|T|$ and
 \begin{align}
   \begin{split}\label{def linear forms}
     \mathcal{L}((m_u)_{u\in U},m_{t+1}) & = m_{t+1}+\sum_{u\in U}\nu_um_u,
     \\ \mathcal{L}_{\k\i}((m_u)_{u\in U},m_{t+1}) & = e_{\kappa\iota}m_{t+1
     }+\sum_{u\in U} N_{u\kappa\iota}m_u,~ \textrm{for}~\k\in [l],
     \i\in I_\k.
 \end{split}
 \end{align}

\begin{remark}
 We denote here by $c_U(\o/\p)$ what is called $c_U(q)$ in
 \cite{SV/14} and \cite{Voll-Functional-equations-annals}. The latter
 notation obscures the fact that this quantity may depend on the prime
 ideal~$\p$ and not just on the cardinality of the residue
 field. (Texts such as \cite{Denef-degree} avoided this pitfall.) We
 are grateful to Tobias Rossmann for pointing this out to us.
\end{remark}

 We uniformly rewrite the functions $\Xi_{U}(q,s)$ in terms of zeta
 functions of polyhedral cones in a fan, in close analogy to the proof
 of {\cite[Lemma 3.1]{duSG/00}}.  We first extend the domain of
 summation in $\Xi_{U}(q,s)$ by taking the sum over all ${\bf
   m}=(m_1,\cdots,m_{t+1})\in\NN^{t}_0\times\NN$ with $m_u=0$ iff
 $u\in T\setminus U$.  Choose a finite triangulation $\{R_i
 \}_{i\in[w]_0}$ of $\bb{R}_{\geq 0}^{t+1}$ consisting of relatively
 open, pairwise disjoint \emph{simple} rational polyhedral cones
 $R_i$, which eliminates the ``min"-terms in the exponent of $q$ in
 $\Xi_{U}(q,{\bf s})$.  (It will be important for future purposes to
 note that this decomposition can be chosen independently of~$\p$,
 depending just on the chosen principalization $(Y,h)$.)  We assume
 that $R_0=\{{\bf 0}\}$ and $R_1,\cdots,R_{z}$ are exactly the
 one-dimensional cones (or rays) in $\{R_i\}_{i\in[w]}$. For every
 $j\in[z]$, write $R_j=\bb{R}_{>0}{\bf r}_j$, where ${\bf
   r}_j\in\NN_0^{t+1}$ is the shortest integral vector on the
 ray~$R_j$.  For every $i\in[w]$ there exists a set $M_i\subseteq[z]$
 of rays such that
  \begin{equation*}
   R_i=\bigoplus_{j\in M_i}\bb{R}_{>0}{\bf r}_j
  \end{equation*}
  (direct sum of monoids). Note that $i\in[z]$ if and only if
  $|M_i|=1$, in which case $M_i=\{i\}$. For $i\in[w]$, the simplicity
  of $R_i$ yields that
  \begin{equation}\label{equ:R_tau_i}
   R_i \cap \NN_0^{t+1} = \bigoplus_{j\in M_i}\bb{N}{\bf r}_j.
  \end{equation} 
 
 For $U\subseteq T$, set
 \begin{equation*}
  \cal{C}_{U}=\{{\bf m}\in\NN_0^{t}\times\NN\mid
  \textrm{$m_u=0$ iff $u\in T\setminus
    U$}\},
 \end{equation*}
the domain of summation in~\eqref{def:Xi_U}. As
$\dot\bigcup_{U\subseteq T} \cal{C}_U = \NN_{0}^t \times \NN$, not all
of the cones $R_i$, $i\in[w]$, are relevant for sums of the
form~\eqref{def:Xi_U}.  Rather, for each $U\subseteq T$ there exists a
uniquely determined subset $W_U'\subseteq [w]$ such that
 \begin{equation}\label{equ:C_tau_U}
  \cal{C}_{U}=\dot\bigcup_{i\in W_U'}R_i\cap \NN_0^{t+1}.
 \end{equation} 
The set $W' := \bigcup_{U\subseteq T}W'_U \subseteq [w]$ indexes the
cones which do not lie in the boundary component $\RR^{t}_{\geq 0}
\times \{0\}$ of $\RR^{t+1}_{\geq 0}$.
 
Let $i\in W'$. Restricting the summation in \eqref{def:Xi_U} to $R_i$
-- or, indeed, its closure --, allows us to rewrite the exponent of
$q$ in this sum 
as $$-\sum_{u=1}^{t+1}\left(\sum_{\k\in[l]}C_{u\k}
s_\k+D_u\right)m_u$$ for suitable $C_{u\k}, D_u\in\NN_{\geq 0}$ for
all $u\in [t+1], \k\in[l]$.  For $j\in[z]$, write $${\bf
  r}_j=(r_{j1},\cdots,r_{j\,t+1})\in\NN_{\geq 0}^{t+1}.$$ For each
$j\in M_i$ and $\k\in[l]$, set
\begin{equation*}
A_{j\k}=\sum_{u=1}^{t+1}r_{ju}
C_{u\k} \quad \textrm{ and } \quad B_j=\sum_{u=1}^{t+1}r_{ju} D_u.
\end{equation*} 
Note that $B_j>0$ for all $j$ as, for $u\in [t]$, we find that
$r_{ju}>0$ iff $D_u>0$. Combining \eqref{equ:C_tau_U} and
\eqref{equ:R_tau_i} yields, for $U\subseteq T$,
 \begin{align}\label{Xi_Ut}
   \Xi_{U}(q,{\bf s})
   &= \sum_{i\in W_U'}\sum_{{\bf m}\in
     R_i\cap(\NN_0^{t}\times\NN)}
   q^{-\sum_{u=1}^{t+1}(\sum_{\k\in[l]}C_{u\k} s_\k +D_u)m_u}\nonumber\\
   &= \sum_{i\in W_U'}\prod_{j\in
     M_i}\frac{q^{-(\sum_{\k\in[l]}A_{j\k}
       s_\k+B_j)}}{1-q^{-(\sum_{\k\in[l]}A_{j\k}
       s_\k+B_j)}}.
 \end{align} 
 Given $i\in W'$, set
 \begin{equation}\label{def-c^t_i} 
 c_i(\o/\p)=c_U(\o/\p)~ \rm{and}~ U_i=U,
 \end{equation}
 where $U\subseteq T$ is the unique subset of $T$ such that $i\in
 W_U'$.
 \begin{proposition}\label{pro:denef.good}
   Assume that $(Y,h)$ has good reduction modulo $\p$. Then
 \begin{equation}\label{equ:denef.general}
 Z({\bf s}) = \frac{(1-q^{-1})^{d+1}}{q^{\binom{d}{2}}}
  \sum_{i\in W'}
 c_i(\o/\p)(q-1)^{|U_i|} \prod_{j\in
   M_i}\frac{q^{-(\sum_{\k\in[l]}A_{j\k}
     s_\k+B_j)}}{1-q^{-(\sum_{\k\in[l]}A_{j\k}
     s_\k+B_j)}}.
 \end{equation}
 \end{proposition}
 \begin{proof}
 Observing that 
$$\dot\bigcup_{U\subseteq T} \dot\bigcup_{i\in W_U'} R_i = \dot\bigcup_{i\in W'}
 R_i = \RR_{\geq 0}^t \times \RR_{>0}$$ and using
 \eqref{equ:Z(s).general.good} and \eqref{Xi_Ut}, we obtain
 \begin{align*}
 Z({\bf s}) & = \frac{(1-q^{-1})^{d+1}}{q^{\binom{d}{2}}}
 \sum_{U\subseteq T}c_{U}(\o/\p)(q-1)^{|U|}\sum_{i\in
   W_U'}\prod_{j\in M_i}\frac{q^{-(\sum_{\k\in[l]}A_{j\k}
     s_\k+B_j)}}{1-q^{-(\sum_{\k\in[l]}A_{j\k} s_\k+B_j)}}
 \\ & = \frac{(1-q^{-1})^{d+1}}{q^{\binom{d}{2}}} 
 \sum_{i\in W'} c_i(\o/\p)(q-1)^{|U_i|} \prod_{j\in
   M_i}\frac{q^{-(\sum_{\k\in[l]}A_{j\k}
     s_\k+B_j)}}{1-q^{-(\sum_{\k\in[l]}A_{j\k} s_\k+B_j)}}.\qedhere
 \end{align*}\end{proof}


 \subsection{Bad reduction} \label{subsec:bad.red} Assume now that $\p$
 is one of the finitely many prime ideals of $\O$ such that the
 principalization $(Y,h)$ does not have good reduction modulo~$\p$.
 In this case, formula \eqref{equ:denef.general} may not
 hold. Informally speaking, this is due to the breakdown of Hensel's
 lemma if the components $E_u$ are singular or have nonnormal
 crossings modulo~$\p$. To a certain extent, this failure can be
 mitigated by analyzing the integral defining $Z(\bfs)$ -- or rather
 its pull-back under the principalization~$h$ -- on cosets modulo
 $\p^N$ for some $N\in\NN$; in the case of good reduction, $N=1$ is
 sufficient. This motivates the following definition, generalizing the
 rational functions $\Xi_U$ introduced in~\eqref{def:Xi_U}. For
 $N\in\NN$, $U\subseteq T$, $(d_{\k\i})\in\NN_0^{\prod_{\k\in[l]}
   I_\k}$, set
 \begin{equation}\label{def:Xi_U.general}
   \Xi_{U,(d_{\kappa\iota})}^N (q,{\bf s}) =
   \sum_{\substack{(m_u)_{u\in U}\in\NN^{|U|}_{\geq
         N},\\ m_{t+1}\in\NN}} q^{-\mathcal{L}((m_u)_{u\in
       U},m_{t+1})-\sum_{\k\in[l]}s_\k \min\{ \mathcal{L}_{\k\iota}((m_u)_{u\in
       U},m_{t+1})-d_{\k\iota} \mid \iota\in I_\k \} }.
 \end{equation}
 Here, $\mathcal{L}$ and $\{\mathcal{L}_{\k\iota}\}_{\k\in[l], \i\in
   I_\k}$ are the linear forms defined in~\eqref{def linear forms}.
 The following is analogous to and proven in the same way as
 {\cite[Corollary 4.2]{Voll-compact-p-adic-analytic-arithmetic}}.
 \begin{proposition}\label{pro:denef.bad}
   There exist $N\in\NN$, finite sets $J\subset\NN_0$ and
   $\Delta\subset \NN_0^{\prod_{\k\in[l]} I_\k}$, all depending on $\p$,
   such that
  \begin{equation}\label{equ:Z(s).general.bad}
    Z({\bf s})=\frac{(1-q^{-1})^{d+1}}{q^{N\binom{d}{2}}}
    \sum_{\substack{U\subseteq T,j\in J\\(d_{\kappa\iota})\in \Delta
    }} c_{U,j,(d_{\kappa\iota})}(\o/\p) (q^N-q^{N-1})^{|U|}q^{-j} \,
    \Xi_{U,(d_{\kappa\iota})}^N (q,{\bf s}),
  \end{equation}
  where each $c_{U,j,(d_{\kappa\iota})}(\o/\p)$ is the number of
  $\o/\p$-rational points of a certain variety over~$\o/\p\cong\bb{F}_q$.
 \end{proposition}
 Notice that \eqref{equ:Z(s).general.bad} generalizes
 \eqref{equ:Z(s).general.good}, the formula for good reduction;
 cf.\ {\cite[Section~4.3]{Voll-compact-p-adic-analytic-arithmetic}}.

 \subsection{Application to Poincar\'e series and local representation zeta functions}\label{subsec:app}

 In this section we argue that all but finitely many of the Euler
 factors $\zeta_{\G(\O_\p)}(s)$ in \eqref{equ:fine.euler} are
 described by (univariable substitutions of) $\p$-adic integrals
 $\cal{Z}_\frak{o}(\rho,\sigma,\tau)$ of the form~\eqref{Zpadic} which
 fit into the general form \eqref{equ:general.Z}. Hence, the
 conclusions of the previous two sections are applicable, allowing for
 a uniform description of these Euler factors. In particular, we will
 show that the real parts of the poles of these local representation
 functions -- and thus, in particular, their abscissae of convergence
 -- are all elements of a finite set ${\bf P}$ of rational numbers,
 depending on $\G$ but independent of~$L$. Moreover, we will show that
 ${\bf P}$ also contains the abscissae of convergence of the finitely
 many ``exceptional'' Euler factors not covered by the generic
 treatment. This will leave us, in Section~\ref{sec:proof.thmA}, free
 to concentrate on analytic properties of the co-finite Euler product
 over the ``generic'' Euler factors.

Note that the bases ${\bf e}$ and ${\bf f}$ defined in
Section~\ref{sec:rep.T.groups}, and thus the matrices $\cal{R}({\bf
  Y})$ and $\cal{S}({\bf Y})$ and the data ${\bf
  b}=(b_1,\cdots,b_d)\in\NN_0^d$, are defined only locally. In order
to give uniform formulae for local zeta functions we need, however, a
global source of local bases. As we do not assume that $\O$ is a
principal ideal domain, we may not hope for a global analogue of the
construction of the bases ${\bf e}$ and ${\bf f}$. Instead, we first
choose an $\O_K$-basis ${\bf f}=(e_{r-k+1},\cdots,e_{r-k+d})$ for a
free finite-index $\O_K$-submodule of the isolator~$\i(\Lambda')$. By
{\cite[Lemma~2.5]{SV/14}}, it can be extended to an $\O_K$-basis ${\bf
  e}$ for a free finite-index $\O_K$-submodule $M$ of~$\Lambda$. Let $p$
be a rational prime.

If $p$ divides neither $|\Lambda:M|$ nor $|\iota(\Lambda'):\Lambda'|$,
then, for every prime ideal $\p$ of $\O$ lying above~$p$, tensoring
the $\O_K$-basis ${\bf e}$ with $\O_\p$ yields an $\O_\p$-basis for
$\Lambda(\O_\p) \cong M(\O_\p)$ as in Section~\ref{sec:rep.T.groups},
with ${\bf b}=(0,\cdots,0)$. (Note that the condition that $p$ does
not divide $|\iota(\Lambda'):\Lambda'|$ is missing in the discussion
of \cite[Section~2.3]{SV/14}. This omission has no bearing on the
proof of \cite[Theorem~A]{SV/14}, as it applies to only finitely many
prime ideals.) If, in addition, $p$ is odd, then the zeta functions of
the pro-$p$ groups $\G(\O_\p) = \exp( \Lambda(\O_\p) )$ are given by a
univariable substitution of the multivariable $\p$-adic integral
$\mathcal{Z}_{\O_\p}(\rho,\sigma,\tau)$ given in~\eqref{Zpadic};
cf.\ Proposition~\ref{pro:poin.p} and \eqref{poincare-I=1}. As
explained in detail in
\cite[Section~4.1.3]{Voll-compact-p-adic-analytic-arithmetic}, this
integral fits into the framework developed in the current section.
The key equation is the following analogue of
\cite[eq.~(4.8)]{Voll-compact-p-adic-analytic-arithmetic}, which
asserts that, for suitable data in \eqref{equ:general.Z} and vectors
$\bfa,\bfb\in\ZZ^l$,
\begin{equation}\label{equ:translate}
\mathcal{Z}_{\O_\p}(-s/2,-1,us+v-d-1) = \frac{1}{\prod_{i=1}^{d-1}(1-q^{-i})}Z(\bfa s + \bfb).
\end{equation}
The fundamental idea is to extend the integral $\mathcal{Z}_{\O_\p}$
trivially to an integral over $\p \times \GL_d(\o)$, interpreting
$(\o^d)^*$ as the space of first columns, say, of matrices in
$\GL_d(\o)$;
cf.\ \cite[Remark~2.1]{Voll-Functional-equations-annals}. The
postulated $B$-invariance holds because, for each $i$ and each~$j$,
the polynomials in the sets $F_j(\bfY)$ and $G_i(\bfY)$ defined in
\eqref{eq:poly F} and \eqref{eq:poly G} are all homogeneous of the
same degree, depending only on $i$ resp.~$j$. The global nature of the
bases ensures that they are defined over~$\O_K$.  Hence the analysis of
Sections~\ref{subsec:good.red} and \ref{subsec:bad.red} applies if $p$
does not divide $|\Lambda:M|$ or~$|\iota(\Lambda'):\Lambda'|$.

If $p$ does divide $|\Lambda:M|$ or $|\iota(\Lambda'):\Lambda'|$, then
$M(\O_\p)$ is a finite-index sublattice of $\Lambda(\O_\p)$ (proper if
and only if $p$ divides $|\Lambda:M|$). Whilst $M(\O_\p)$ may not even
be a Lie lattice, there exists a constant $C=C(\p)\in\NN_0$ such that,
for all $m\geq C$, the Lie lattice $\p^m M(\O_\p)$ gives rise, by
$p$-adic Lie theory, to a finitely generated nilpotent pro-$p$ group
$\H^m(\O_\p) := \exp(\p^m M(\O_\p))$ to which the Kirillov orbit
method is applicable. Indeed, it suffices to ensure that $\H^m(\O_\p)$
is uniformly powerful; cf.\ {\cite[Theorem~2.12]{Jaikin-zeta}}.  We
entertain no hope to describe the zeta functions of the groups
$\H^m(\O_\p)$ explicitly in the given generality, but can still make
statements about their abscissae of convergence. Note that, as all of
the groups $\H^m(\O_p)$ have finite index in $\G(\O_\p)$, the
abscissae of convergence $\alpha(\H^m(\O_\p))$ and $\alpha(\G(\O_\p))$
all coincide; cf.\ Lemma~\ref{abscissa commensurable}. We achieve
control over the abscissa of convergence by way of expressing the zeta
functions of the groups $\H^m(\O_\p)$ in terms of Poincar\'e series
akin to~\eqref{Poincare define}. It may also be expressed in terms of
$\p$-adic integrals such as~\eqref{Zpadic}, but with the families of
($\O_K$-)polynomials $F_j(\bfY)$ and $G_i(\bfY)$ replaced by
$\pi^{C(F_j,m)}F_j(\bfY)$ and $\pi^{C(G_i,m)}G_i(\bfY)$ for suitable
$C(F_j,m), C(G_i,m)\in\NN_0$. Whilst the explicit formulae for these
``deformed'' $\p$-adic integrals may differ substantially from the
generic formulae, their abscissae of convergence may be bounded in
terms of the formulae of Denef type describing the generic case.

Let $(Y,h)$ be a principalization for the $\O_K$-ideal
$\cal{I}=\prod_{j\in[u]}({F}_j(\bfY))\prod_{i\in[v]}({G}_i(\bfY))$. We
continue to use the notation developed in this section so far in this
specific context. Recall, in particular, from
Section~\ref{subsec:good.red} the choice of triangulation
$\{{R_i}\}_{i\in[w]_0}$ of $\RR^{t+1}_{\geq 0}$ consisting of
relatively open, pairwise disjoint simple rational polyhedral
cones~$R_i$, $i\in[w]$, and the associated notation. The following is
a direct consequence of Proposition~\ref{pro:denef.good}.
 \begin{proposition}\label{cor:denef.poincare.good} 
 Assume that $(Y,h)$ has good reduction modulo $\p$ and that $\p$ does
 not divide $|\Lambda:M|$ or $|\iota(\Lambda'):\Lambda'|$.  There
 exist $A_j,B_j\in\QQ$, for $j\in[z]$, such that the Poincar\'{e}
 series \eqref{Poincare define} satisfies
 \begin{equation}\label{possible poles}
   \cal{P}_{\cal{R},\cal{S},\O_\p}(s) =
   1+\frac{(1-q^{-1})^d}{q^{\binom{d}{2}}\prod_{i=1}^{d-1}(1-q^{-i})}
   \sum_{i\in W'}c_i(\o/\p) (q-1)^{|U_i|} \prod_{j\in
     M_i}\frac{q^{-(A_j s+B_j)}}{1-q^{-(A_j s+B_j)}}.
 \end{equation}
 \end{proposition}

 \begin{proof} 
  Applying equations~\eqref{poincare-I=1}, \eqref{equ:translate}, and
  \eqref{equ:denef.general}, we obtain
 \begin{align*}
 \cal{P}_{\cal{R},\cal{S},\O_\p}(s)
&=1+(1-q^{-1})^{-1}\cal{Z}_{\O_\p}(-s/2,-1,us+v-d-1))\nonumber \\ &
 =
 1+(1-q^{-1})^{-1}\frac{1}{\prod_{i=1}^{d-1}(1-q^{-i})}Z(\bfa s + \bfb)
 \\ &=   1+\frac{(1-q^{-1})^d}{q^{\binom{d}{2}}\prod_{i=1}^{d-1}(1-q^{-i})}
   \sum_{i\in W'}c_i(\o/\p) (q-1)^{|U_i|} \prod_{j\in
     M_i}\frac{q^{-(A_j s+B_j)}}{1-q^{-(A_j s+B_j)}}.\nonumber
 \end{align*}
 for suitable rational numbers $A_j$ and $B_j$. 
 \end{proof}

\begin{remark}\label{rem:divisible}
 By way of construction of the relevant integral, the numbers
 $c_i(\o/\p)$ are all divisible by
 $q^{\binom{d-1}{2}}\prod_{i=1}^{d-1}\frac{(1-q^{-i})}{1-q^{-1}} =
 |\GL_d(\Fq)/B(\Fq)| / |\mathbb{P}^{d-1}(\Fq)|$. This is due to the
 fact that we are not using the full generality of integrals of the
 form \eqref{equ:general.Z}, but restrict to integrands which only
 involve the matrices' first columns.
\end{remark}

\begin{remark}\label{rem:A.j}
 The numbers $A_j$ are positive for all $j\in [z]\cap W'$. By
 definition, the set $W'$ comprises exactly those $i$ with the
 property that there exists $j\in M_i$ with $j\in W'$. Geometrically
 speaking, this conditions means that at least one of the boundary
 rays of the simple cone $R_i$ does not lie on the boundary component
 $\RR_{\geq 0}^t\times \{0\}$ of the positive orthant. Consequently,
 $\sum_{j\in M_i}A_j > 0$ for $i \in W'$.
\end{remark}

 \begin{corollary}\label{Cor:poles superset}
  Assume that $(Y,h)$ has good reduction modulo $\p$ and that $p$ does
  not divide $|\Lambda:M|$ or $|\iota(\Lambda'):\Lambda'|$.  The real
  parts of the poles of $\cal{P}_{\cal{R},\cal{S},\O_\p}(s)$ are
  elements of the finite set
   \begin{equation}\label{P} {\bf
       P}=\left\{\frac{-B_j}{A_j}\suchthat j\in[z] \cap W',
      \right\} \subset \QQ.
 \end{equation}
 \end{corollary}

 \begin{proof}
   This follows from inspection of~\eqref{possible poles}.
 \end{proof}

Combining formulae~\eqref{poincare-I=1}
and~\eqref{equ:Z(s).general.bad} for the specific integral under
consideration yields -- at least in principal -- explicit formulae of
Denef type for the Poincar\'e series also in the case of bad
reduction, albeit not as concise as the one in
Proposition~\ref{cor:denef.poincare.good}. We have, however, some
control over their poles.

\begin{proposition}\label{pro:poincare.poles.in.P}
 For \emph{every} nonzero prime ideal $\p$ of $\O$ such that $p$ does
 not divide $|\Lambda:M|$ or $|\iota(\Lambda'):\Lambda'|$, the real
 parts of the poles and thus, in particular, the abscissa of
 convergence of the Poincar\'{e} series
 $\cal{P}_{\cal{R},\cal{S},\O_\p}(s)$ are elements of ${\bf P}$.
\end{proposition}
\begin{proof}
  By Corollary~\ref{Cor:poles superset} we may concentrate on the case
  that the principalization $(Y,h)$ does not have good reduction
  modulo $\p$.  As in the generic case, the Poincar\'e series
  $\cal{P}_{\cal{R},\cal{S},\O_\p}(s)$ is obtained from the
  multivariate zeta function $\mathcal{Z}_{\o}(\rho,\sigma,\tau)$
  defined in \eqref{Zpadic} upon an affine linear substitution of
  variables; cf.\ \eqref{poincare-I=1}. In
  Proposition~\ref{pro:denef.bad} we gave a formula of Denef type for
  such multivariable $\p$-adic integrals in terms of the
  series~$\Xi_{U,(d_{\kappa\iota})}^N(q,{\bf s})$; see
  \eqref{def:Xi_U.general}.
  \cite[Proposition~4.5]{Voll-compact-p-adic-analytic-arithmetic}
  asserts that, for any ${\bf a},{\bf b}\in\ZZ^l$, the real parts of
  the poles of the univariable functions $\Xi_{U,(d_{\kappa\iota})}^N
  (q,{\bf a}s + {\bf b})$ are independent of $q$, $N$, and
  $(d_{\kappa\iota})$. The claim follows.
 \end{proof}

\begin{proposition}\label{local poles}
 For every nonzero prime ideal $\p$ of $\O$, the abscissa of
 convergence of $\zeta_{\G(\O_\p)}(s)$ is an element of ${\bf P}$.
\end{proposition}
\begin{proof}
We first consider the case that $p$ does not divide $2|\Lambda:M| |
\iota(\Lambda'):\Lambda'|$. By Proposition~\ref{pro:poin.p}, in this
case the zeta function $\zeta_{\G(\O_\p)}(s)$ is given by the
Poincar\'e series $\mathcal{P}_{\cal{R},\cal{S},\O_\p}(s)$. The
abscissa of convergence of the latter is an element of ${\bf P}$ by
Proposition~\ref{pro:poincare.poles.in.P}.

Assume now that $p$ does divide $2|\Lambda:M| |
\iota(\Lambda'):\Lambda'|$. In this case, $M(\O_\p)$ may not even be a
Lie lattice, but there exists $C=C(\p)\in\NN_0$ such that, for all
$m\geq C$, the finitely generated nilpotent pro-$p$ group $\H^m(\O_\p) :=
\exp(\p^m M(\O_\p))$ is amenable to the Kirillov orbit method. Fix
such an~$m$. The zeta function $\zeta_{\H^m(\O_\p)}(s)$ is
expressible, via a Poincar\'e series, in terms of a $\p$-adic integral
of the form~\eqref{Zpadic}. In fact, the analysis of
Section~\ref{subsec:poincare} is applicable, with an $\O_\p$-basis for
$\p^m M(\O_p)$ obtained by multiplying ${\bf e}\otimes \O_\p$ by
$\pi^m$. The relevant $\p$-adic integral \eqref{Zpadic} differs from
the generic one considered in the previous case essentially in that
the families of polynomials $F_j({\bf Y})$ and $G_i({\bf Y})$ are
replaced by suitable $\pi$-power multiples thereof,
viz.\ $\pi^{C(F_{j},m)}F_j({\bf Y})$ and $\pi^{C(G_i,m)}G_i({\bf Y})$
for $C(F_j,m), C(G_i,m)\in\NN_0$. Arguing as in the proof of
Proposition~\ref{pro:poincare.poles.in.P}, one sees that this leaves
unaffected the abscissa of convergence of the relevant univariable
substitution. It follows that the abscissa of convergence of the
relevant Poincare\'e series, and thus of $\zeta_{\H^m(\O_\p)}(s)$ is
an element of ${\bf P}$. But $\alpha(\H^m(\O_\p)) = \alpha(\G(\O_\p))$
by Lemma~\ref{abscissa commensurable}, so the result follows.
\end{proof}

\section{Proof of Theorem~\ref{thmA}}\label{sec:proof.thmA}
Let $K$ be a number field with ring of integers $\O_K$ and $\Lambda$
be a nilpotent $\O_K$-lattice of nilpotency class $c$ as in
Section~\ref{sec:rep.T.groups}.  Let further $L$ be a finite extension
of $K$, with ring of integers $\O=\O_L$. The group
$\G(\O)=\G_\Lambda(\O)$ is a $\cal{T}$-group of nilpotency class~$c$;
cf.\ {\cite[Section~2]{SV/14}}. Recall the Euler
product~\eqref{equ:fine.euler}.  Choose an $\O_K$-submodule $M$ of
$\Lambda$ of finite index in $\Lambda$ as in Section~\ref{subsec:app}
and let $(Y,h)$ be a principalization for the $\O_K$-ideal
$\cal{I}=\prod_{j\in[u]}\left({F}_j(\bfY)\right)\prod_{i\in[v]}\left({G}_i(\bfY)\right)$,
where ${F}_j(\bfY)$ and ${G}_i(\bfY)$ are the finite sets of
polynomials defined in \eqref{eq:poly F} and~\eqref{eq:poly
  G}. (Recall that they are defined over $\O_K$ as the local bases are
obtained from an $\O_K$-basis for~$M$.) Let $Q$ be the finite set of
prime ideals $\p$ of $\O$ satisfying the following:
\begin{itemize}
 \item[$\circ$] the residue characteristic $p$ of $\p$ divides
   $|\Lambda:M| |\iota(\Lambda'):\Lambda'|$,
 \item[$\circ$] $(Y,h)$ does not have good reduction modulo~$\p$, and
 \item[$\circ$] $(p,c)=(2,3)$.
\end{itemize}

We fix a finite triangulation $\{R_i\}_{i\in[w]_0}$ of $\bb{R}_{\geq
  0}^{t+1}$ as in Section~\ref{sec:denef} and recall the notation
introduced there. Given a prime ideal $\p\notin Q$ and $i\in W'$ we
write, in the notation of Proposition~\ref{cor:denef.poincare.good},
 \begin{equation*}
 Z_{i,\p}(s)=\frac{(1-q^{-1})^d}{\prod_{i=1}^{d-1}(1-q^{-i})}
 c_i(\o/\p)q^{-\binom{d}{2}}(q-1)^{|U_i|}
 \prod_{j\in M_i}\frac{q^{-(A_j s+B_j)}}{1-q^{-(A_j s+B_j)}}.
 \end{equation*}
Note that the $Z_{i,\p}(s)$ are ordinary generating functions in
$q^{-s}$ with nonnegative coefficients. Propositions~\ref{pro:poin.p}
and~\ref{cor:denef.poincare.good} amount to the statement that
 \begin{equation}\label{equ:euler.not.Q}
  \prod_{\p\notin Q}\zeta_{\G(\O_\p)}(s) =
 \prod_{\p\notin Q} \left( 1+\sum_{i\in W'}
   Z_{i,\p}(s) \right).
 \end{equation}

 \subsection{Proof of (1)}\label{subsec:proof.of.(1)}
 Recall that, for \emph{every} $\p\in\Spec(\O)$, the abscissa of
 convergence $\alpha(\G(\O_\p))$ is an element of the finite set~${\bf
   P}$ of rational numbers defined in~\eqref{P};
 cf.\ Proposition~\ref{local poles}.  To prove the first part of
 Theorem~\ref{thmA} it thus suffices to show that the abscissa of
 convergence of $\prod_{\p\notin Q}\zeta_{\G(\O_\p)}(s)$ is a rational
 number which is strictly larger than $\max \bfP$ and only depends
 on~$\bfG$. Equivalently, it suffices to prove that the abscissa of
 convergence of
 $$\sum_{i\in W'}\left(\sum_{\p\notin Q}Z_{i,\p}(s)\right)$$
 is such a rational number.

 Recall that $T$ is the finite set of irreducible components $E_u$ of
 the pre-image under $h$ of the subvariety of $\GL_d/B$ defined
 by~$\cal{I}$.  Let $U\subseteq T$ and $\p\notin Q$. Recall from
 \eqref{def-c^t_i} that, given  $i\in W'$ with
 $U=U_i$, if $i\in W_U'$, then $c_i(\o/\p)=c_U(\o/\p)$, the
 number of $\o/\p$-rational points of
 $\overline{E_U}\setminus\cup_{V\varsupsetneq U}\overline{E_V}$.  Let
 $d_U$ be the dimension of $E_U=\cap_{u\in U}E_u$. Note that
 $d_U=\binom{d}{2}-|U|$ since the $E_u$, $u\in T$, are hypersurfaces
 in a $\binom{d}{2}$-dimensional variety intersecting with normal
 crossings; cf.~\cite[Proposition~4.13]{duSG/00}.

 Let $\{F_{U,b}\}_{b\in I_U}$ be the irreducible components over $K$
 of $E_U$ of maximal dimension~$d_U$.  For $b\in I_U$ let
 $l_\p(F_{U,b})$ be the number of irreducible components of
 $\overline{F_{U,b}}$ over $\o/\p$ which are absolutely irreducible
 over an algebraic closure of~${\o/\p}$. We record the following
 consequence of the Lang-Weil estimate.

\begin{proposition}\cite[Proposition~4.9]{duSG/00}\label{pro:lang-weil} 
  There exists $C\in\RR_{>0}$ such that, for all $U\subseteq T$ and
  $\p\not\in Q$,
 $$ \vert c_U(\o/\p) - \sum_{b\in I_U}l_\p(F_{U,b})q^{d_U} \vert \leq
 C \cdot q^{d_U-1/2}.$$ Moreover, for each $b\in I_U$,
 $l_\p(F_{U,b})>0$ for a set of prime ideals of positive density.
\end{proposition}

Note that the finitely many primes excluded in
\cite[Proposition~4.9]{duSG/00} are those for which $E_I \setminus
\bigcup_{J\supseteq I}E_J$ has bad reduction; cf.\ the proof of
\cite[Lemma~4.7]{duSG/00}. Such prime ideals have been excluded from
our discussion by the definition of~$Q$.

\begin{corollary}\label{cor:lang-weil}
 For any sequence $(r_\p)_{\p\notin Q}$ of rational numbers,
 $\sum_{\p\not\in Q}c_U(\o/\p)q^{-d_U}r_\p$ converges absolutely if
 and only if $\sum_{\p\not\in Q}\sum_{b\in I_U}l_\p(F_{U,b})r_\p$
 converges absolutely.
\end{corollary}
Similar to \cite[Lemma 4.11 and Corollary 4.14]{duSG/00} one uses
Corollary~\ref{cor:lang-weil} to show that, given 
$i\in W'$, the abscissa of convergence $\alpha_{i}$ of
$\sum_{\p\notin Q}Z_{i,\p}(s)$ -- or, equivalently, of
$\prod_{\p\not\in Q} (1 + Z_{i,\p}(s))$ -- is
 $$\a_{i}=\max\left\{ \frac{1-\sum_{k\in M_i}B_k}{\sum_{k\in M_i}A_k},
-\frac{B_j}{A_j}\suchthat j\in M_i, A_j \neq 0 \right\}\in\bb{Q}.$$
Recall that $\sum_{k\in M_i}A_k\neq 0$ for $i\in W'$;
cf.\ Remark~\ref{rem:A.j}.  If $i\in [z]$, then $M_i
=\{i\}$. Thus $$\a_{i}=\frac{1-B_i}{A_i}\textrm{ for }i\in[z]\cap W'.$$
As $\alpha(\G(\O_\p))\in{\bf P}$ for all $\p\in\Spec(\O)$, it follows
that
 \begin{multline*}
  \alpha(\G(\O)) = \max\{\{\alpha_{i} \mid i\in W'\}\cup
\{\alpha(\G(\O_\p)) \mid \p\in Q \} \} =\max\{\alpha_{i} \mid
i\in W'\} \in\bb{Q}.
\end{multline*} Similar to {\cite[Corollary
    4.14, Lemma 4.15]{duSG/00}} we obtain that
 \begin{equation*}
 \aO 
  =  \max\left\{\frac{1-B_i}{A_i}
 \suchthat  i\in [z]\cap W'
 \right\},
 \end{equation*}
 and $\alpha_{i}<\alpha(\G(\O))$ if~$i>z$. Note that $\aO>\max{\bf
   P}$. Evidently, $a(\G) := \alpha(\G(\O))$ is independent
 of~$\O$. This concludes the proof of~(1).
 \subsection{Proofs of (2) and the last part}
 In order to prove part (2) of Theorem~\ref{thmA}, it suffices to
 prove the analogous statement for the Euler product
 \eqref{equ:euler.not.Q} whose abscissa of convergence is $a(\G)$, as
 we have shown in Section~\ref{subsec:proof.of.(1)}. Indeed, the
 finitely many remaining Euler factors of $\zeta_{\G(\O)}(s)$ all have
 abscissa of convergence strictly smaller than $a(\G)$ and, as nonzero
 generating functions in $q^{-s}$ with nonnegative coefficients, do
 not vanish in a neightbourhood of~$s=a(\G)$. To establish meromorphic
 continuation of \eqref{equ:euler.not.Q} to
 $\{s\in\bb{C}\mid\rea(s)>a(\G)-\delta\}$ for some $\delta>0$, only
 depending on $\bfG$, we proceed as in {\cite[Theorem
   4.16]{duSG/00}}. For $\p\notin Q$, set
 \begin{align}
   \mathcal{R} & = \left\{ i\in [z] \cap W' \suchthat
   \frac{1-B_i}{A_i}=a(\G) \right\}, \label{Omega.tau}\\ V_\p(s) & =
   \prod_{i\in \mathcal{R}}\left(1-c_i(\o/\p)q^{-d_{U_i}}q^{-(A_i
     s+B_i)}\right).\nonumber
 \end{align}
Informally speaking, the set $\mathcal{R}$ records which of the rays $R_i$,
$i\in[z]$, contribute to the global abscissa of convergence. The
inverses of the functions $V_\p$ will serve as ``approximations'' of
the Euler factors $\zeta_{\G(\O_\p)}(s)$ for $\p\not\in Q$.

 The proof of (2) will be accomplished once we produce constants
 $\delta_1,\delta_2\in\RR_{>0}$, both depending only on $\G$, such
 that the following hold:
 \begin{itemize}
 \item[(I)] The product $\prod_{\p\notin Q}V_\p(s)$ has abscissa of
   convergence $a(\G)$ and meromorphic continuation to
   $\{s\in\bb{C}\mid\rea(s)>a(\G)-\delta_1\}$. 
 \item[(II)] The Euler product $\prod_{\p\notin Q}\left(
     1+\sum_{i\in W'}Z_{i,\p}(s)\right)
   V_\p(s)$ converges on $\{s\in\bb{C}\mid\rea(s)>a(\G)-\delta_2\}$.
 \end{itemize}
 Given $i\in \mathcal{R}$, set
 $$V_{i}(s):=\prod_{\p\notin
   Q}\left(1-c_i(\o/\p)q^{-d_{U_i}}q^{-(A_i s+B_i)}\right),$$
 so
 $$\prod_{\p\notin Q} V_\p(s)=\prod_{i\in \mathcal{R}}V_{i}(s).$$ In order
 to prove (I), it suffices to show that, for each $i\in \mathcal{R}$, the
 function $V_{i}(s)$ defines a meromorphic function on
 $\{s\in\bb{C}\mid\rea(s)>a(\G)-\delta^{(i)}\}$ for some $\delta^{(i)}>0$
 depending only on~$\bfG$.  Fix $i\in \mathcal{R}$, set $U=U_i$, and let
 $\{F_{U,b}\}_{b\in I_U}$ be as in Section~\ref{subsec:proof.of.(1)}.
 Set $$\widetilde{V}_{i}(s)=\prod_{b\in I_U}\prod_{\p\notin
   Q}\left(1-l_\p(F_{U,b})q^{-(A_i s+B_i)}\right).$$ The following
 result follows from the straightforward generalization of
 \cite[Lemma~4.6]{duSG/00} to arbitrary number fields and the fact
 that the constant $\delta$ in this lemma can be chosen to
 be~$\frac{1}{2}$.

 \begin{proposition}\label{pro:artin} 
  The function $\widetilde{V}_{i}(s)$ converges on $\left\{ s\in \CC
  \mid \rea(s) > \frac{1-B_i}{A_{i}}\right\}$ and can be continued to
  a meromorphic function on $\left\{s\in\bb{C}\mid\rea(s)> \frac{1 -
    2B_i}{2A_i}\right\}$. 
 \end{proposition}

 Note that $\frac{1-2B_i}{2A_i} < \frac{1-B_i}{A_i}$ as
 $A_i>0$. The following is a variation of
 notation introduced in the proof of {\cite[Theorem~4.16]{duSG/00}}.

 \begin{definition}
   Let $(F_\p(s))_{\p\not\in Q}$ and $(G_\p(s))_{\p\not\in Q}$ be
   families of functions in a complex variable $s$.  Given
   $\Delta\in\RR_{>0}$, we say that the respective Euler products are
   $\Delta$-equivalent, written $$\prod_{\p\notin Q}
   F_\p(s)\equiv_\Delta\prod_{\p\notin Q} G_\p(s),$$ if
   $\sum_{\p\notin Q}(F_\p(s)-G_\p(s))$ converges on
   $\{s\in\bb{C}\mid\rea(s)>\Delta\}$.
 \end{definition}

The Lang-Weil estimate (Proposition~\ref{pro:lang-weil}) yields that
there exist uniform (i.e.\ independent of $\p$) constants
$\Delta^{(i)}\in\QQ_{>0}$ with $\Delta^{(i)} < \alpha_{i}(=
(1-B_i)/A_i)$ such that $V_{i}(s)
\equiv_{\Delta^{(i)}}\widetilde{V}_{i}(s)$.  Set
 $$\delta_1 = a(\G)-\max\left\{\frac{1- 2B_i}{2A_i},
\Delta^{(i)}\suchthat i\in \mathcal{R}\right\}>0.$$ It follows that
$\prod_{\p\notin Q}V_\p(s)$ is a meromorphic function on
$\{s\in\bb{C}\mid\rea(s)>a(\G)-\delta_1\}$, which proves~(I).

Set $$d_1 =
\max\{0,\a_{i}\mid i\in([z]\cap W')\setminus \mathcal{R}\} < a(\G).$$
Then, by the definition~\eqref{Omega.tau} of $\mathcal{R}$,
 $$\prod_{\p\notin Q}\left( 1+\sum_{i\in
   W'}Z_{i,\p}(s)\right)\equiv_{d_1}\prod_{\p\notin Q}\left(
 1+\sum_{i\in \mathcal{R}}Z_{i,\p}(s)\right).$$ Recall that if
 $i\in\mathcal{R}$, then $M_i=\{i\}$ and thus
 $$Z_{i,\p}(s)=\frac{(1-q^{-1})^d}{\prod_{i=1}^{d-1}(1-q^{-i})} c_i(\o/\p)q^{-\binom{d}{2}}(q-1)^{|U_i|}
 \frac{q^{-(A_i s+B_i)}}{1-q^{-(A_i s+B_i)}}.$$ For $\p\notin Q$, and
 $i\in \mathcal{R}$, set
 $$\overline{Z_{i,\p}(s)}=c_i(\o/\p)q^{-d_{U_i}}q^{ -(A_is+B_i)}.$$
There exists a uniform constant $d_2\in\QQ_{>0}$ with $d_2 < a(\G)$
such that
$$\prod_{\p\notin Q}\left(
1+\sum_{i\in
  \mathcal{R}}Z_{i,\p}(s)\right)\equiv_{d_2}
\prod_{\p\notin Q}\left( 1+\sum_{i\in
  \mathcal{R}}\overline{Z_{i,\p}(s)}\right).$$ Set
 $$d_3= \max\left\{d_1,d_2, \frac{1-\sum_{k\in J}B_k}{\sum_{k\in
    J}A_k}, \frac{1-2B_j}{2A_j} \suchthat j\in \mathcal{R}, J\subseteq
\mathcal{R}, \vert J \vert \geq 2 \right\} < a(\G).$$
Then
$$\prod_{\p \not\in Q}V_\p(s) \equiv_{d_3} \prod_{\p\not\in Q}\left(1-\sum_{i\in\mathcal{R}}\overline{Z_{i,\p}(s)}\right),$$
and consequently
 \begin{align*}
 \prod_{\p\notin Q}\left( 1+\sum_{i\in W'}Z_{i,\p}(s)\right) V_\p(s)
 & \equiv_{d_3} \prod_{\p\notin Q}\left( 1+\sum_{i\in
   \mathcal{R}}Z_{i,\p}(s)\right) V_\p(s)\\& \equiv_{d_3} \prod_{\p\notin
   Q}\left( 1+\sum_{i\in \mathcal{R}}\overline{Z_{i,\p}(s)}\right) \left(
 1-\sum_{i\in \mathcal{R}}\overline{Z_{i,\p}(s)}\right).
 \end{align*}
The product on the right hand side converges on
$\{s\in\bb{C}\mid\rea(s)> d_3\}$. Hence (II) follows.

It is clear that $\delta(\G) := \max\{a(\G)-\delta_1,d_3 \}$ is
independent of the number field $K$ and the zeta function
$\zeta_{\G(\O)}(s)$ may be continued
to~$\{s\in\bb{C}\mid\rea(s)>a(\G)-\delta(\G) \}$. The order
$\beta(\G)$ of the pole of $\prod_{\p\not\in Q} V_\p^{-1}(s)$ at
$s=a(\G)$ is clearly an invariant of~$\G$, which proves~$(2)$.  The
statement about the holomorphy of the continued function follows from
the fact that the meromorphic continuation was achieved by writing
$\zeta_{\G(\O)}(s)$ as a product of translates of Artin $L$-functions
(implicit in the proof of Proposition~\ref{pro:artin}) and a Dirichlet
series convergent on $\{s\in \CC \suchthat \rea(s) > a(\G) -
\delta(\G)\}$; cf.\ proof of \cite[Corollary~4.22]{duSG/00}. The last
part of Theorem~\ref{thmA} follows from the Tauberian theorem
{\cite[Theorem 4.20]{duSG/00}}.  This completes the proof of Theorem
\ref{thmA}.

\section{Proof of Corollary~\ref{cor}}\label{sec:proof.thmB}

Let $G$ be a finitely generated nilpotent group of nilpotency class
$c$ and Hirsch length~$h$, say, and $G_0 \leq G$ a torsion-free
subgroup of finite index in $G$. By the general theory of the Mal'cev
correspondence between $\T$-groups and torsion-free nilpotent Lie
rings, there exists an $h$-dimensional $\QQ$-Lie algebra
$L_{G_0}(\QQ)$ associated to ${G_0}$, together with an injective map
$\log : {G_0}\to L_{G_0}(\QQ)$ such that $\log({G_0})$ rationally
spans $L_{G_0}(\QQ)$; cf.\ {\cite[Chapter 6]{Segal-polycyclic}}.  If
$H$ is a subgroup of finite index in~${G_0}$, then
$L_H(\QQ)=L_{G_0}(\QQ)$.  In general, $\log({G_0})$ is not even an
additive subgroup of~$L_{G_0}(\QQ)$.  There exists, however, a
subgroup $H$ of finite index in ${G_0}$, and hence in $G$, such that
$\log(H)$ is a $\ZZ$-Lie lattice inside $L_{G_0}(\QQ)$ satisfying
$\log(H)'\subseteq c!\log(H)$; cf.~{\cite[Theorem~4.1]{GSS88}}. Hence
$H$ may be viewed as the group $\G(\ZZ)$ of $\ZZ$-points of the group
scheme $\G$ associated to the $\ZZ$-Lie lattice $\log(H)$.  Since $H$
is of finite index in~$G$, it follows from
{\cite[Corollary~4.14]{Snocken}} that $\a(G)=\a(H)$.  Moreover,
$\zeta_{G,p}(s)=\zeta_{H,p}(s)$ for all primes $p$ which do not divide
the index~$|G:H|$.  The Euler products $\zeta_{G}(s) = \prod_{p
  \textrm{ prime}}\zeta_{G,p}(s)$ and $\zeta_{H}(s) = \prod_{p
  \textrm{ prime}}\zeta_{H,p}(s)$ therefore coincide apart from
finitely many Euler factors. Notice that, for every prime~$p$, the
abscissae of convergence of $\zeta_{G,p}(s)$ and $\zeta_{H,p}(s)$
concide (cf.~Lemma~\ref{abscissa commensurable}) and are strictly
smaller than the abscissa of convergence~$\alpha(H)$. All Euler
factors are nonzero rational generating functions in $p^{-s}$ (cf.\
\cite[Theorem~1.5]{Hrushovski-Martin}) with nonnegative
coefficients. It follows that they do not vanish in a neighbourhood of
$s=\alpha(G)$. Corollary~\ref{cor} follows now by applying
Theorem~\ref{thmA} to the $\ZZ$-Lie lattice~$\log(H)$.

\bibliographystyle{abbrv}
\def\cprime{$'$} \def\cprime{$'$} \def\cprime{$'$}

\end{document}